\documentclass[a4paper]{article}
\usepackage[T1]{fontenc}
\usepackage{amsfonts,amsmath,amsthm,amssymb}
\usepackage{calc}
\usepackage{gensymb}
\usepackage[nottoc, notlof, notlot]{tocbibind}
\usepackage{fullpage}
\usepackage{mathrsfs}  
\usepackage{bigints}
\usepackage{indentfirst}
\usepackage{enumerate}
\usepackage{authblk}

\DeclareMathOperator{\tr}{Tr}
\DeclareMathOperator{\cov}{Cov}
\DeclareMathOperator{\var}{Var}

\DeclareMathOperator{\supp}{Supp}
\let\limsup\relax
\DeclareMathOperator*{\limsup}{limsup}
\let\liminf\relax
\DeclareMathOperator*{\liminf}{liminf}

\newcommand{\norm}[1]{\left\Vert #1\right\Vert}

\begin{document}

\newtheorem{theorem}{Theorem} [section]
\newtheorem{prop}[theorem]{Proposition} 
\newtheorem{defi}[theorem]{Definition} 
\newtheorem{exe}[theorem]{Exemple} 
\newtheorem{lemma}[theorem]{Lemma} 
\newtheorem{rem}[theorem]{Remark} 
\newtheorem{cor}[theorem]{Corollary} 
\newtheorem{conj}[theorem]{Conjecture}
\renewcommand\P{\mathbb{P}}
\newcommand\E{\mathbb{E}}
\newcommand\N{\mathbb{N}}
\newcommand\1{\mathbf{1}}
\newcommand\C{\mathbb{C}}
\newcommand\CC{\mathcal{C}}
\newcommand\M{\mathbb{M}}
\newcommand\R{\mathbb{R}}
\newcommand\U{\mathbb{U}}
\newcommand\A{\mathcal{A}}
\newcommand\B{\mathcal{B}}
\renewcommand\i{\mathbf{i}}
\renewcommand\S{\mathcal{S}_{N,t}}
\renewcommand\L{\mathcal{L}_{\Phi}}
\renewcommand\d{\partial_i}
\renewcommand\.{\ .}
\renewcommand\,{\ ,}

\def\etc{,\dots ,}

\title{On the operator norm of non-commutative polynomials in deterministic matrices and iid GUE matrices}

\date{}

\author[1]{Beno\^it Collins}
\author[2]{Alice Guionnet}
\author[3]{F\'elix Parraud}
\affil[1]{\small Department of Mathematics, Graduate School of Science, Kyoto University, Kyoto 606-8502, Japan.}
\affil[2]{\small Universit\'e de Lyon, CNRS, ENSL, 46 all\'ee d'Italie, 69007 Lyon.}
\affil[3]{\small Universit\'e de Lyon, ENSL, UMPA, 46 all\'ee d'Italie, 69007 Lyon. Department of Mathematics, Graduate School of Science, Kyoto University, Kyoto 606-8502, Japan.}

\maketitle

\begin{abstract}
		
	Let $X^N = (X_1^N\etc X^N_d)$ be a d-tuple of $N\times N$ independent GUE random matrices and $Z^{NM}$ be any family of 
	deterministic matrices in $\M_N(\C)\otimes \M_M(\C)$. Let $P$ be a self-adjoint non-commutative polynomial. A seminal work of 
	Voiculescu shows that the empirical measure of the eigenvalues of $P(X^N)$  converges towards a deterministic measure defined 
	thanks to  free probability theory. Let now $f$ be a smooth function, the main technical result of this paper is a precise bound of the 
	difference between the expectation of
	$$ \frac{1}{MN} \tr_{\M_N(\C)}\otimes\tr_{\M_M(\C)}\left( f(P(X^N\otimes I_M,Z^{NM})) \right) \, $$
	
	\noindent and its limit when $N$ goes to infinity. If $f$ is six times differentiable, we show that it is bounded by 
	$M^2 \norm{f}_{\mathcal{C}^6} N^{-2}$. As a corollary we obtain a new proof of a result of Haagerup and Thorbj\o rnsen, later developed 
	by Male,  which gives sufficient conditions for the operator norm of a polynomial evaluated in $(X^N,Z^{NM},{Z^{NM}}^*)$ to converge 
	almost surely towards its free limit. Restricting ourselves to polynomials in independent GUE matrices, we give concentration estimates 
	on the largest eingenvalue of these polynomials around their free limit. A direct consequence of these inequalities is that there exists 
	some $\beta>0$ such that for any $\varepsilon_1 < (3+\beta)^{-1}$ and $\varepsilon_2 < 1/4$, almost surely for $N$ large enough,
	
	$$ - \frac{1}{N^{\varepsilon_1}}\ \leq \| P(X^N)\| - \norm{P(x)} \leq\ \frac{1}{N^{\varepsilon_2}} \. $$
	
	\noindent Finally if $X^N$ and $Y^{M_N}$ are independent and $M_N = o(N^{1/3})$, then almost surely, the norm of any polynomial in $(X^N\otimes I_{M_N}, I_N\otimes Y^{M_N})$ converges almost surely towards its free limit. This result is an improvement of a Theorem of Pisier in \cite{pisier}, who was himself using estimates from Haagerup and Thorbj\o rnsen, where $M_N$ had size $o(N^{1/4})$.
	
\end{abstract}

\section{Introduction}

Given several deterministic matrices whose spectra are known, the spectra of  a non-commutative polynomial evaluated in these matrices is 
not well defined since it depends as well on the eigenvectors of these matrices. If one takes these vectors at random, 
 it is possible to get some surprisingly good results, in particular when the dimension of these matrices goes to infinity. Indeed, the limit can 
 then be computed thanks to free probability. 
 This theory  was  introduced by Voiculescu in the early nineties as a non-commutative probability theory equipped with a notion of freeness 
 analogous to independence in classical probability theory.  Voiculescu showed that this theory was closely related with 
 Random Matrix Theory in a seminal paper  \cite{Vo91}. He considered independent  matrices taken from the Gaussian Unitary Ensemble 
 (GUE), which are random matrix is an $N\times N$ self-adjoint random matrix whose distribution is proportional to the measure 
 $\exp\left( -N/2\tr_N(A^2)\right) dA$, where $dA$ denotes the Lebesgue measure on the set of $N\times N$ Hermitian matrices. We refer to 
 Definition \ref{GUEdef} for a more precise statement.   Voiculescu proved that given $X_1^N,\dots,X_d^N$ independent GUE matrices, the 
 renormalized trace of a polynomial $P$ evaluated in these matrices converges towards a deterministic limit $\alpha(P)$. 
 Specifically, the following holds true almost surely:

\begin{equation}\label{dv} \lim_{N\to \infty} \frac{1}{N}\tr_N\left( P(X_1^N,\dots,X_d^N) \right) = \alpha(P) \.\end{equation}

\noindent Voiculescu computed the limit $\alpha(P)$ with the help of free probability.  If $A_N$ is a self-adjoint matrix of size $N$, 
then one can define the empirical measure of its (real) eigenvalues by 

$$ \mu_{A_N} = \frac{1}{N} \sum_{ i=1}^N \delta_{\lambda_i} \,$$

\noindent where $\delta_{\lambda}$ is the Dirac mass in $\lambda$ and $\lambda_1\etc \lambda_N$ are the eingenvalue of $A_N$. In 
particular, if $P$ is a self-adjoint polynomial, that is such that for any self adjoint matrices $A_1\etc A_d$, $P(A_1\etc A_d)$ is a self-adjoint 
matrix, then one can define the random measure $\mu_{P(X_1^N,\dots,X_d^N)}$. In this case, Voiculescu's result \eqref{dv} implies that   
there exists a measure $\mu_P$ with compact support such that almost surely $\mu_{P(X_1^N,\dots,X_d^N)}$ converges weakly towards 
$\mu_P$ : it is given by $\mu_P(x^k)=\alpha(P^k)$ for all integer numbers $k$.\\

However, the convergence of the empirical measure of the eigenvalues of a matrix does not say anything about the local properties  of its 
spectrum, in particular about the convergence of the norm of this matrix, or the local fluctuations of its spectrum. However, when dealing with 
a single matrix, incredibly precise results are known. 
For exemple it is well-known that the largest eigenvalue of a GUE random matrix converges almost surely towards $2$. More precisely, if 
$X_N$ is a GUE random matrix of size $N$, then almost surely

$$ \lim_{N\to \infty} \norm{X_N} = 2 \. $$

\noindent The proof, for the more general case of a Wigner matrix with entries with finite moments,  was proved in \cite{monomat}. This result 
was obtained  under the optimal assumption that their   fourth moment is finite \cite{BY}. Concerning the GUE, much more precise results were obtained by Tracy and Widom in the early nineties in \cite{TW1}. The main result of their paper is the existence of a 
continuous decreasing 
function $F_2$ from $\R$ to $(0,1)$ such that with $\lambda_1(X^N)$ the largest eingenvalue of $X^N$,

$$ \lim_{N\to \infty} P\big(N^{2/3} (\lambda_1(X^N) - 2) \geq s\big) = F_2(s) \.$$
This was recently generalized to Wigner matrices \cite{Sosh,erdoyau, TV, LY} up to optimal hypotheses. One can as well study the 
localization of the eigenvalues in the bulk as well as their fluctuations \cite{ESY, erdoyau}. 

On the other hand, there are much less results available when one deals with a polynomial in several random matrices. In fact, up to today,  the only local 
fluctuations results concern perturbative polynomials \cite{FiGu} or local laws \cite{EKN} under some assumptions which are shown to hold 
for homogeneous polynomials of degree two.  However, impressive progress was made 
in  2005 by  Haagerup and Thorbj\o rnsen \cite{HT}: they proved the almost sure convergence 
of the norm of a polynomial evaluated in independent GUE matrices. For $P$ a self-adjoint polynomial, they proved that almost surely, for 
$N$ large enough, 
\begin{equation}
\label{spec}
	\sigma\left( P(X_1^N,\dots,X_d^N) \right) \subset \supp \mu_P + (-\varepsilon,\varepsilon) \,
\end{equation}

\noindent where $\sigma(H)$ is the spectrum of $H$ and $\supp \mu_P$ the support of the measure $\mu_P$. This is equivalent to saying 
that for any polynomial $P$, $\norm{P(X_1^N,\dots,X_d^N)}$ converges almost surely towards 
$\sup \left\{ |x|\ \middle| x\in \supp \mu_P \right\}$ (see proposition \ref{hausdorff}). The result \eqref{spec} was a major breakthrough in the 
context of free probability and was refined in multiple ways. In \cite{SCH}, Schultz used the method of \cite{HT} to prove the same result with 
Gaussian orthogonal or symplectic matrices instead of Gaussian unitary matrices. In \cite{CD}, Capitaine and Donati-Martin proved it for 
Wigner matrices under some technical hypothesis on the law of the entries. This result itself was then extended by Anderson in \cite{AND} to 
remove most of the technical assumption. In \cite{male}, Male made a conceptual improvement to the result of Haagerup and Thorbj\o rnsen, 
by allowing to work both with GUE and deterministic matrices. Finally, Collins and Male proved in \cite{collins_male} the same result with 
unitary Haar matrices instead of GUE matrices by using Male's former paper.

With the exception of \cite{collins_male}, all of these results are essentially based on the method introduced by Haagerup and Thorbj\o rnsen. 
Their first tool is called the linearization trick. The main idea is that given a polynomial $P$,  the spectrum of $P(X_1^N,\dots,X_d^N)$ is 
closely related to  the spectrum of 
$$ L_N = a_0\otimes I_N + \sum_{ i=1}^d a_i\otimes X^N_i \,$$

\noindent where $a_0\etc a_d$ are matrices of size $k$ depending only on $P$. Thus we trade a polynomial of degree $d$ with coefficient in 
$\C$ by a polynomial of degree $1$ with coefficient in $\M_{k(d)}(\C)$. The second idea to understand the spectrum of $L_N$ is to study a 
quantity similar to its Stieljes transform. It is defined on $\C\setminus \R\times\{0\}$ by

$$ G_{L_N} : z\mapsto \frac{1}{kN} \left(\tr_k\otimes\tr_N\right)\left(\left(L_N- \left( \begin{matrix}
z & 0 \\
0 & I_{k-1}
\end{matrix} \right)\otimes I_N \right)^{-1}\right) \.$$

\noindent The subsequent technical steps depend on the model of random matrix. The aim is to study $G_{L_N}$ for $z$ whose imaginary part as small as possible, that is of order $N^{-c}$ for some constant $c$.

An issue of this method is that it does not give easily good quantitative estimates. One aim of this paper is to remedy to this 
problem. We develop a new method that allows us to give a new proof of the main theorem of Male in \cite{male}, and thus a new proof of the 
result of Haagerup and Thorbj\o rnsen. Our approach requires neither the linearization trick, nor the study of the Stieljes transform and 
attacks the problem directly. In this sense the proof is more direct and less algebraic. We will apply it to a generalization of GUE matrices by  
tackling the case of GUE random matrices tensorized with deterministic matrices.

A usual strategy to study outliers, that are the eigenvalues going away from the spectrum,  is to study the \emph{non-renormalized} trace of 
smooth non-polynomial functions evaluated in independent GUE matrices i.e. if $P$ is self-adjoint:
$$\tr_N\left( f(P(X_1^N,\dots,X_d^N)) \right) \.$$
This strategy was also  used by Haagerup, Thorbj\o rnsen and Male,,
\noindent Indeed it is easy to see that if $f$ is a function which takes value $0$ on $(-\infty,C-\varepsilon]$, $1$ on $[C,\infty)$ and in $[0,1]$ 
elsewhere, then

$$ \P\Big( \lambda_1(P(X_1^N,\dots,X_d^N)) \geq C \Big) \leq\ \P\Big( \tr_{N}\left( f(P(X_1^N,\dots,X_d^N)) \right) \geq 1 \Big) $$

\noindent Hence, if we can prove that $\tr_{N}\left( f(P(X_1^N,\dots,X_d^N)) \right)$ converges towards $0$ in probability, this would already 
yield expected results. The above is just a well-known exemple, but one can get much more out of this strategy. 
Therefore, we need to study the non-
renormalized trace. The case where $f$ is a polynomial function has already been studied a long time ago, starting with the pioneering works 
\cite{CDu,GTCL}, and later formalized by the concept of second order freeness \cite{MingoSpeicher}. However here we have to deal with a 
function $f$ which is at best $C^{\infty}$. This makes things considerably more difficult and forces us to adopt a completely different 
approach. The main result is the following Theorem (for the notations, we refer to Section \ref{definit} -- for now, let us specify that 
$\frac{1}{N}\tr_N$ denotes the usual renormalized trace on $N\times N$ matrices whereas $\tau$ denotes its free limit):

\begin{theorem}
	\label{imp1}
	Let the following objects be given,
	\begin{itemize}
		\item $X^N = (X_1^N,\dots,X_d^N)$ independent $GUE$ matrices in $\M_N(\C)$,
		\item $x = (x_1,\dots,x_d)$ a system of free semicircular variable,
		\item $Z^{NM} = (Z_1^{NM},\dots,Z_q^{NM})$ deterministic matrices in $\M_N(\C)\otimes \M_M(\C)$,
		\item $P\in \C\langle X_1,\dots,X_{d+2q}\rangle_{sa}$ a self-adjoint polynomial,
		\item $f\in \mathcal{C}^6(\R)$.
	\end{itemize}
	
	\noindent Then there exists a polynomial $L_P$ which only depends on $P$ such that for any $N,M$,
	
	\begin{align*}
	\Bigg| &\E\left[\frac{1}{MN}\tr_{MN}\Big(f\left(P\left(X^N\otimes I_M,Z^{NM},{Z^{NM}}^*\right)\right)\Big)\right] - \tau_N\otimes\tau_M\Big(f\left(P\left(x\otimes I_M,Z^{NM},{Z^{NM}}^*\right)\right)\Big) \Bigg| \\
	&\leq \frac{M^2}{N^2} \norm{f}_{\mathcal{C}^6} L_P\left(\norm{Z^{NM}}\right) \,
	\end{align*}
	
	\noindent where $\norm{f}_{\mathcal{C}^6}$ is the sum of the supremum on $\R$ of the first six derivatives. Besides if 
	$Z^{NM} = (I_N\otimes Y_1^M,\dots ,I_N\otimes Y_q^M)$ and that these matrices commute, then we have the same inequality without 
	the $M^2$.
	
\end{theorem}
This theorem is a consequence of the slightly sharper, but less explicit, Theorem \ref{imp2}. It is essentially the same statement, but instead 
of having the norm $C^6$ of $f$, we have the fourth moment of the Fourier transform of $f$. The above Theorem calls for a few   remarks.

\begin{itemize}
	\item We assumed that the matrices $Z^{NM}$ were deterministic, but thanks to Fubini's Theorem we can assume that they are random 
	matrices as long as they are independent from $X^N$. In this situation though, $L_P\left(\norm{Z^{NM}}\right)$ in the right side of the 
	inequality is a random variable (and thus we need some additional assumptions if we want its expectation to be finite for instance).
	\item In Theorems \ref{imp1} and \ref{imp2} we have $X^N\otimes I_M$ and $x\otimes I_M$, however it is very easy to replace them by 
	$X^N\otimes Y^M$ and $x\otimes Y^M$ for some matrices $Y^M_i\in M_M(\C)$. Indeed we just need to apply Theorem \ref{imp1} or 
	\ref{imp2} with $Z^{NM}=I_N\otimes Y^M$. Besides, in this situation, $L_P\left(\norm{Z^{NM}}\right) =  L_P\left(\norm{Y^M}\right)$ does 
	not depend on $N$. What this means is that if we have a matrix whose coefficients are polynomial in $X^N$, and that we replace $X^N$ 
	by $x$, we only change the spectra of this matrix by $M^2N^{-2}$ in average.
	\item Unfortunately we cannot get rid of the $M^2$ in all generality. The specific case where we can is when 
	$Z^{NM} = (I_N\otimes Y_1^M,\dots ,I_N\otimes Y_q^M)$, where the $Y_i^M$ commute: this shows that the $M^2$ term is really a 
	non-commutative feature. 
\end{itemize}

A detailed overview of the proof is given in Subsection \ref{overview}. The main idea of the proof is to interpolate $GUE$ matrices and a free 
semicircular system with the help of a free Ornstein-Uhlenbeck process. For a reference, see \cite{BianeSpeicher2001}. When using this 
process, the Schwinger-Dyson equations, which can be seen as an integration by part, appear in the computation. For more information 
about these equations we refer to Proposition \ref{SD}. As we will see, they will play a major role in this paper. Theorem \ref{imp1} is the crux 
of the paper and allows us to deduce many corollaries. Firstly we rederive a new proof of the following theorem. The first statement is 
basically Theorem 1.6 from \cite{male}. The second one is an improvement of Theorem 7.8 from \cite{pisier} on the size of the tensor  from 
$N^{1/4}$ to $N^{1/3}$. This theorem is about strong convergence of random matrices, that is the convergence of the norm of polynomials in 
these matrices, see definition \ref{freeprob}.

\begin{theorem}
	\label{strongconv}
	Let  the  following objects be given:
	\begin{itemize}
		\item $X^N = (X_1^N,\dots,X_d^N)$ independent $GUE$ matrices of size $N$,
		\item $x = (x_1,\dots,x_d)$ a system of free semicircular variable,
		\item $Y^M = (Y_1^M,\dots,Y_p^M)$ random matrices of size $M$, which almost surely, as $M$ goes to infinity, converges strongly 
		in distribution towards a $p$-tuple $y$ of non-commutative random variables in a $\mathcal C^*$- probability space $\B$ with a 
		faithful trace $\tau_{\B}$.
		\item $Z^N = (Z_1^N,\dots,Z_q^N)$ random matrices of size $N$, which almost surely, as $N$ goes to infinity, converges strongly 
		in distribution towards a $q$-tuple $z$ of non-commutative random variables in a $\mathcal C^*$- probability space  with a faithful 
		trace,
	\end{itemize}
	
	\noindent then the following holds true.
	
	\begin{itemize}
		\item If $X^N$ and $Z^N$ are independent, almost surely, $(X^N, Z^N)$ converges strongly in distribution towards 
		$\mathcal{F} = (x,z)$, where $\mathcal{F}$ belongs to a $\mathcal C^*$- probability space $(\A,*,\tau_{\A},\norm{.})$ in which $x$ 
		and $z$ are free.
		\item If $X^N$ and $Y^{M_N}$ are independent and $M_N = o(N^{1/3})$, almost surely, 
		$(X^N\otimes I_{M_N}, I_N\otimes Y^{M_N})$ converges strongly in distribution towards $\mathcal{F} = (x\otimes 1, 1\otimes y)$. 
		The family $\mathcal{F}$ thus belongs to $\A\otimes_{\min}\B$ (see definition \ref{mini}). Besides if the matrices $Y^{M_N}$ 
		commute, then we can weaken the assumption on $M_N$ by only assuming that $M_N = o(N)$.
	\end{itemize}
	
\end{theorem}

As we discussed earlier, understanding the Stieljes transform of a matrix gives a lot of information about its spectrum. This was actually a 
very important point in the proof of Haagerup and Thorbj{\o}rnsen's Theorem. Our proof does not use this tool, however our final result, 
Theorem \ref{imp2}, allows us to deduce the following estimate with sharper constant than what has previously been done. Being given a self-
adjoint $NM\times NM$ matrix, we denote by $G_A$ its Stieltjes transform:
$$G_A(z)=\frac{1}{NM}\tr_{NM}\left(\frac{1}{z-A}\right)\.$$
This definition extends to the tensor product of free semi-circular variables by replacing $\tr_{NM}$ by $\tau_N\otimes \tau_M$. 

\begin{cor}
	\label{stieljes}
	
	Given
	\begin{itemize}
		\item $X^N = (X_1^N,\dots,X_d^N)$ independent $GUE$ matrices of size $N$,
		\item $x = (x_1,\dots,x_d)$ a system of free semicircular variable,
		\item $Y^M = (Y_1^M,\dots,Y_p^M,{Y_1^M}^*,\dots,{Y_p^M}^*)$ deterministic matrices of size $M$ a fixed integer and their 
		adjoints,
		\item $P\in \C\langle X_1,\dots,X_d,Y_1,\dots,Y_{2p}\rangle_{sa}$ a self-adjoint polynomial,
	\end{itemize}
	\noindent there exists a polynomial $L_P$ such that for every $Y^M$, $z\in \C\backslash \R$, $N\in \N$,
	
	$$ \left| \E\left[ G_{P(X^N\otimes I_M, I_N\otimes Y^M)}(z) \right] - G_{P(x\otimes I_M, I_N\otimes Y^M)}(z) \right| \leq L_P\left(\norm{Y^M}\right) \frac{M^2}{N^2} \left( \frac{1}{ \left| \Im(z)\right|^5} + \frac{1}{ \left| \Im(z)\right|^2} \right) \.$$
\end{cor}

One of the limitation of Theorem \ref{imp1} is that we need to pick $f$ regular enough. Actually by approximating $f$, we can afford to take 
$f$ less regular at the cost of a slower speed of convergence. In other words, we trade some degree of regularity on $f$ for a smaller 
exponent in $N$. The best that we can achieve is to take $f$ Lipschitz. Thus it makes sense to introduce the Lipschitz-bounded metric. This 
metric is compatible with the topology of the convergence in law of measure. Let $\mathcal{F}_{LU}$ be the set of Lipschitz function from 
$\R$ to $\R$, uniformly bounded by $1$ and with Lipschitz constant at most $1$, then

$$ d_{LU}(\mu,\nu) = \sup_{f\in \mathcal{F}_{LU}} \left| \int_{\R} f d\mu - \int_{\R} f d\nu \right| \.$$

\noindent For more information about this metric we refer to Annex C.2 of \cite{alice}. In this paper, we get the following result:

\begin{cor}
	\label{LU}
	
	Under the same notations as in Corollary \ref{stieljes}, there exists a polynomial $L_P$ such that for every matrices $Y^M$ and $M,N\in \N$,	
	$$ d_{LU}\left(\E[\mu_{P(X^N\otimes I_M, I_N\otimes Y_M)}] , \mu_{P(x\otimes I_M, I_N\otimes Y_M)}\right) \leq L_P\left(\norm{Y^M}\right) \frac{M^2}{N^{1/3}} \.$$
	
\end{cor}

One of the advantage of Theorem \ref{imp1} over the original proof of Haagerup and Thorbj\o rnsen is that if we take $f$ which depends on 
$N$, we get sharper estimates in $N$. For exemple if we assume that $g$ is a $\mathcal{C}^{\infty}$ function with bounded support, as we 
will see later in this paper we like to work with $f: x\mapsto g(N^{\alpha}x)$ for some constant $\alpha$. Then its $n$-th derivative will be of 
order $N^{n\alpha}$. In the original work of Haagerup, Thorbj\o rnsen (see \cite{HT}, Theorem 6.2) the eighth derivative appears for the 
easiest case where our polynomial $P$ is of degree $1$, and the order is even higher in the general case. But in Theorem \ref{imp1} the sixth 
derivative appears in the general case. Actually if we look at the sharper Theorem \ref{imp2}, the fourth moment of the Fourier transform 
appears, which is roughly equivalent to the fourth derivative for our computations. This allows us to compute an estimate of the difference 
between $\E\left[ \norm{P(X^N\otimes I_M, I_N\otimes Y^M)} \right]$ and its limit. To do that, we use Proposition \ref{tail}  from 
\cite[Theorem 1.1]{SS15} which implies that if we denote by 
$\mu_{P(x\otimes I_M, 1\otimes Y^M)} $ the spectral measure of $P(x\otimes I_M, 1\otimes Y^M)$, then there exists $\beta\in\R^+$ such that
	
	\begin{equation}\label{taileq} \limsup\limits_{\varepsilon\to 0}\quad \varepsilon^{-\beta}\mu_{P(x\otimes I_M, 1\otimes Y^M)} \left( \left[\norm{P(x\otimes I_M, 1\otimes Y^M)}-\varepsilon, \norm{P(x\otimes I_M, 1\otimes Y^M)}\right] \right) > 0 \. \end{equation}
With the help of standard measure concentration estimates, we then get the following Theorem:

\begin{theorem}
	\label{concentration}
	We consider
	\begin{itemize}
		\item $X^N = (X_1^N,\dots,X_d^N)$ independent $GUE$ matrices of size $N$,
		\item $x = (x_1,\dots,x_d)$ a system of free semicircular variable,
		\item $Y^M = (Y_1^M,\dots,Y_p^M)$ deterministic matrices of size $M$ a fixed integer and their adjoints.
	\end{itemize}
	
	 Almost surely, for any polynomial $P\in \C\langle X_1,\dots,X_{d},Y_1,\dots,Y_p\rangle$, there exists constants $K$ and $C$ such that for any $\delta>0$,
	
	\begin{equation}
	\label{upper}
	\P\left( N^{1/4} \left( \norm{P(X^N\otimes I_M, I_N\otimes Y^M)} - \norm{P(x\otimes I_M, 1\otimes Y^M)} \right) \geq \delta + C \right) \leq e^{-K \delta^2 \sqrt{N}} + d e^{-N}\,
	\end{equation}
	
	\begin{equation}
	\label{lower}
	\P\left( N^{1/(3+\beta)} \left( \norm{P(X^N\otimes I_M, I_N\otimes Y^M)} - \norm{P(x\otimes I_M, 1\otimes Y^M)} \right) \leq -\delta - C \right) \leq e^{-K \delta^2 N^{\frac{1+\beta}{3+\beta}}} + d e^{-N} \.
	\end{equation}
	
\end{theorem}

This theorem is interesting because of its similarity with Tracy and Widom's result about the tail of the law of the largest eingenvalue of a GUE 
matrix. We have smaller exponent in $N$, and thus we can only show the convergence towards $0$ with exponential speed, however we are 
not restricted to a single GUE matrix, we can chose any polynomial evaluated in GUE matrices. Besides by applying Borel-Cantelli's Lemma, 
we immediately get:

\begin{theorem}
	\label{meso}
	We consider
	\begin{itemize}
		\item $X^N = (X_1^N,\dots,X_d^N)$ independent $GUE$ matrices of size $N$,
		\item $x = (x_1,\dots,x_d)$ a system of free semicircular variable,
		\item $Y^M = (Y_1^M,\dots,Y_p^M)$ deterministic matrices of size $M$ a fixed integer and their adjoints.
	\end{itemize}
	
	\noindent Then almost surely, for any polynomial $P\in \C\langle X_1,\dots,X_{d},Y_1,\dots,Y_p\rangle$, there exists a constant 
	$c(P)>0$ such that for any $c(P)>\alpha>0$,
	
	$$ \lim_{N\to \infty} N^{\alpha} \Big| \norm{P(X^N\otimes I_M, I_N\otimes Y^M)} - \norm{P(x\otimes I_M, 1\otimes Y^M)} \Big| =0 \.$$
	Moreover, if $\beta$ satisfies \eqref{taileq},  then almost surely for any $\alpha< (3+\beta)^{-1}$ and $\varepsilon < 1/4$, for $N$ large 
	enough,
	
	$$ -N^{-\alpha} \leq \norm{P(X^N\otimes I_M, I_N\otimes Y^M)} - \norm{P(x\otimes I_M, 1\otimes Y^M)} \leq N^{-\varepsilon} \. $$	
	
\end{theorem}

In order to conclude this Introduction, we would like to say that while it is not always easy to compute the constant $\beta$ in all generality, it 
is possible for some polynomials. In particular, if our polynomial is evaluated in a single GUE matrix, then the computation are heavily 
simplified by the fact that we know the distribution of a single semicircular variable, and we can always process them. Finally, the constant 
$(3+\beta)^{-1}$ is clearly a worst case scenario and can be easily improved if $\beta$ is explicit.

This paper is organised as follows. In Section \ref{definit}, we recall the definitions and properties of free probability, non-commutative 
calculus and Random Matrix Theory needed for this paper. Section \ref{mainsec} contains the proof of Theorem \ref{imp1}. And finally in 
Section \ref{proofcoro} we give the proof of the remaining Theorem and Corollaries.

\section{Framework and standard properties}
\label{definit}

\subsection{Usual definitions in free probability}
\label{deffree}

In order to be self-contained, we begin by reminding the following definitions from free probability.

\begin{defi}
	\label{freeprob}
	\begin{itemize}
		\item A $\mathcal{C}^*$-probability space $(\A,*,\tau,\norm{.}) $ is a unital $\mathcal{C}^*$-algebra $(\A,*,\norm{.})$ endowed with 
		a state $\tau$, i.e. a linear map $\tau : \A \to \C$ satisfying $\tau(1_{\A})=1$ and $\tau(a^*a)\geq 0$ for all $a\in \A$. In this paper 
		we always assume that $\tau$ is a trace, i.e. that it satisfies $\tau(ab) = \tau(ba) $ for any $a,b\in\A$. An element of $\A$ is called a 
		(non commutative) random variable. We will always work with faithful trace, namely, for $a\in\A$, $\tau(a^*a)=0$ if and only if 
		$a=0$. In this case the norm is determined by $\tau$ thanks to the formula:
		$$ \norm{a} = \lim_{k\to\infty} \big(\tau\big((a^*a)^{2k}\big)\big)^{1/2k}. $$
		
		\item Let $\A_1,\dots,\A_n$ be $*$-subalgebras of $\A$, having the same unit as $\A$. They are said to be free if for all $k$, for all 
		$a_i\in\A_{j_i}$ such that $j_1\neq j_2$, $j_2\neq j_3$, \dots , $j_{k-1}\neq j_k$:
		$$ \tau\Big( (a_1-\tau(a_1))(a_2-\tau(a_2))\dots (a_k-\tau(a_k)) \Big) = 0. $$
		Families of non-commutative random variables are said to be free if the $*$-subalgebras they generate are free.
		
		\item Let $ A= (a_1,\ldots ,a_k)$ be a $k$-tuple of random variables. The joint distribution of the family $A$ is the linear form 
		$\mu_A : P \mapsto \tau\big[ P(A, A^*) \big]$ on the set of polynomials in $2k$ non commutative indeterminates. 
		By convergence in distribution, for a sequence of families of variables $(A_N)_{N\geq 1} = (a_{1}^{N},\ldots ,a_{k}^{N})_{N\geq 1}$ 
		in $\mathcal C^*$-algebras $\big( \mathcal A_N, ^*, \tau_N, \norm{.} \big)$,
		we mean the pointwise convergence of
		the map 
		$$ \mu_{A_N} : P \mapsto \tau_N \big[ P(A_N, A_N^*) \big],$$
		and by strong convergence in distribution, we mean convergence in distribution, and pointwise convergence
		of the map
		$$P \mapsto \big\| P(A_N, A_N^*) \big\|.$$
		
		\item A family of non commutative random variables $ x=(x_1,\dots ,x_p)$ is called 
		a free semicircular system when the non commutative random variables are free, 
		selfadjoint ($x_i=x_i^*$, $i=1 \dots p$), and for all $k$ in $\N$ and $i=1 , \dots , p$, one has
		\begin{equation*}
		\tau( x_i^k) =  \int t^k d\sigma(t),
		\end{equation*}
		with $d\sigma(t) = \frac 1 {2\pi} \sqrt{4-t^2} \ \mathbf 1_{|t|\leq2} \ dt$ the semicircle distribution.
		
	\end{itemize}
	
\end{defi}

The strong convergence of   non-commutative random variables is actually equivalent to the convergence of the spectrum of their 
polynomials for the Hausdorff distance. More precisely we have the following proposition whose proof can be found in \cite[Proposition 2.1]
{collins_male} :

\begin{prop}
	\label{hausdorff}
	Let $\mathbf x_N = (x_1^N,\dots, x_p^N)$ and $\mathbf x = (x_1,\dots, x_p)$ be  $p$-tuples of variables in $\mathcal C^*$-probability 
	spaces, $(\mathcal A_N, .^*, \tau_N, \| \cdot \|)$ and $(\mathcal A, .^*, \tau, \| \cdot \|)$, with faithful states. Then, the following assertions 
	are equivalent.
	
	\begin{itemize}
		\item $\mathbf x_N$ converges strongly in distribution to $\mathbf x$.
		\item For any self-adjoint variable $h_N=P(\mathbf x_N)$, where $P$ is a fixed polynomial, $\mu_{h_N}$ converges in weak-$*$ 
		topology to $\mu_h$ where $h=P( \mathbf x)$. Weak-$*$ topology means relatively to continuous functions on $\mathbb C$. 
		Moreover, the spectrum of $h_N$ converges in Hausdorff distance to the spectrum of $h$, that is, for any $\varepsilon >0$, there 
		exists $N_0$ such that for any $N\geq N_0$, 
		\begin{equation}
			\sigma(h_N) \ \subset \ \sigma(h) \ + (-\varepsilon,\varepsilon).
		\end{equation}
	\end{itemize}
	
	\noindent In particular, the strong convergence in distribution of a single self-adjoint variable is equivalent to  its convergence in 
	distribution together with the Hausdorff convergence of its spectrum.
	
\end{prop}

It is important to note that thanks to Theorem 7.9 from \cite{nica_speicher_2006}, that we recall below, one can consider free version of any 
random variable.

\begin{theorem}
	\label{freesum}
	
	Let $(\A_i,\phi_i)_{i\in I}$ be a family of $\mathcal{C}^*$-probability spaces such that the functionals $\phi_i : \A_i\to\C$, $i\in I$, are 
	faithful traces. Then there exist a $\mathcal{C}^*$-probability space $(\A,\phi)$ with $\phi$ a faithful trace, and a family of norm-
	preserving unital $*$-homomorphism $W_i: \A_i\to\A$, $i\in I$, such that:
	
	\begin{itemize}
		\item $\phi \circ W_i = \phi_i$, $\forall i \in I$.
		\item The unital $\mathcal{C}^*$-subalgebras form a free family in $(\A,\phi)$.
	\end{itemize}
\end{theorem}

Let us now fix a few notations concerning the spaces and traces that we use in this paper.

\begin{defi}
\begin{itemize}
	\item $(\A_N,\tau_N)$ is the free sum of $\M_N(\C)$ with a system of $d$ free semicircular variable, this is the $\mathcal{C}^*$-
	probability space built in Theorem \ref{freesum}. Note that when restricted to $\M_N(\C)$, $\tau_N$ is just the regular renormalized trace 
	on matrices. The restriction of $\tau_{N}$ to the $\mathcal{C}^*$-algebra generated by the free semicircular system $x$ is denoted as 
	$\tau$.
	\item $\tr_N$ is the non-renormalized trace on $\M_N(\C)$.
	\item $\M_N(\C)_{sa}$ is the set of self adjoint matrix of $\M_N(\C)$.
	\item We regularly identify $\M_N(\C)\otimes \M_k(\C)$ with $\M_{kN}(\C)$ through the isomorphism 
	$E_{i,j}\otimes E_{r,s} \mapsto E_{i+rN,j+sN} $, similarly we identify $\tr_N\otimes\tr_k$ with $\tr_{kN}$.
	\item If $A^N=(A_1^N,\dots,A_d^N)$ and $B^M=(B_1^M,\dots,B_d^M)$ are two vectors of random matrices, then we denote 
	$A^N\otimes B^M=(A_1^N\otimes B^M_1,\dots,A_d^N\otimes B^M_d)$. We typically use the notation $X^N\otimes I_M$ for the vector 
	$(X^N_1\otimes I_M,\dots,X^N_1\otimes I_M)$.
\end{itemize}
\end{defi}

\subsection{Non-commutative polynomials and derivatives}
\label{poly}

We set $\A_{d,q}=\C\langle X_1,\dots,X_d,Y_1,\dots,Y_q,Y_1^*,\dots,Y_q^*\rangle$ the set of non-commutative polynomial in $d+2q$ 
indeterminates. We endow this vector space with the norm

\begin{equation}
	\label{normA}
	\norm{P}_A = \sum_{M \text{monomial}} |c_M(P)| A^{\deg M} \,
\end{equation}

\noindent where $c_M(P)$ is the coefficient of $P$ for the monomial $M$ and $\deg M$ the total degree of $M$ (that is the sum of its degree 
in each letter $X_1,\dots,X_d,Y_1,\dots,Y_q,Y_1^*,\dots,Y_q^*$). Let us define several maps which we use frequently in the sequel
First, for $A,B,C\in \A_{d,q}$, let

$$ A\otimes B \# C = ACB \,$$

$$ A\otimes B \tilde{\#} C = BCA \,$$

$$ m(A\otimes B) = BA \.$$

\begin{defi}
	\label{application}
	
	If $1\leq i\leq d$, one defines the non-commutative derivative $\partial_i: \A_{d,q} \longrightarrow \A_{d,q} \otimes \A_{d,q}$  by its value 
	on a monomial  $M\in \A_{d,q}$  given by
	$$ \partial_i M = \sum_{M=AX_iB} A\otimes B \,$$
	and then extend it by linearity to all polynomials. Similarly one defines the cyclic derivative  $D_i: \A_{d,q} \longrightarrow \A_{d,q}$ for $P\in \A_{d,q}$ by
	
	$$ D_i P = m\circ \partial_i P \ . $$

\end{defi}

The map $\partial_i$ is called the non-commutative derivative. It is related to Schwinger-Dyson equation on semicircular variable thanks to the following property \ref{SDE}. One can find a proof of the first part in \cite{alice}, Lemma 5.4.7. As for the second part it is a direct consequence of the 
first one which can easily be verified by taking $P$ monomial and then concluding by linearity.

\begin{prop}
\label{SDE}
	Let $ x=(x_1,\dots ,x_p)$ be a free semicircular system, $y = (y_1,\dots,y_q)$ be non-commutative random variable free from $x$, if the 
	family $(x,y)$ belongs to the $\mathcal{C}^*$-probability space $(\A,*,\tau,\norm{.}) $, then for any $P\in \A_{d,q}$,
	$$ \tau(P(x,y,y^*)\ x_i) = \tau\otimes\tau(\partial_i P(x,y,y^*))\ .$$
	
	\noindent Moreover, one can deduce that if $Z^{NM}$ are matrices in $\M_N(\C)\otimes \M_M(\C)$ that we view as a subspace of 
	$\A_N\otimes \M_M(\C)$, then for any $P\in \A_{d,q}$,
	$$ \tau_N\otimes\tau_M\Big( P(x\otimes I_M,Z^{NM},{Z^{NM}}^*)\ x_i\otimes I_M \Big) = \tau_M\Big( (\tau_N\otimes I_M) \bigotimes (\tau_N\otimes I_M) \big(\partial_i P(x\otimes I_M,Z^{NM},{Z^{NM}}^*)\big)\Big) \ . $$
	
\end{prop}

We define an involution $*$ on $\A_{d,q}$ such that

$$(X_i)^* = X_i\, \quad (Y_i)^* = Y_i^*\,\quad (Y_i^*)^* = Y_i$$

\noindent and then we extend it to $\A_{d,q}$ by the formula $(\alpha P Q)^* = \overline{\alpha} Q^* P^*$. $P\in \A_{d,q}$ is said to be self-
adjoint if $P^* = P$. Self-adjoint polynomials have the property that if $x_1,\dots,x_d,z_1,\dots,z_q$ are elements of a $\mathcal{C}^*$-
algebra such as $x_1,\dots,x_d$ are self-adjoint, then so is $P(x_1,\dots,x_d,z_1,\dots,z_q,z_1^*,\dots,z_q^*)$. Now that we have defined the 
notion of self-adjoint polynomial we give a property which justifies computations that we will do later on:

\begin{prop}
	Let the following objects be given,
	\begin{itemize}
		\item $x=(x_1,\dots ,x_p)$ a free semicircular system ,
		\item $X^N = (X_1^N,\dots,X_d^N)$ self-ajoint matrices of size $N$,
		\item $X_t^N = e^{-t/2} X^N + (1-e^{-t})^{1/2} x$ elements of $\A_N$,
		\item $Z^{NM}$ matrices in $\M_N(\C)\otimes \M_M(\C)$,
		\item $f\in \mathcal{C}^0(\R)$,
		\item $P$ a self-adjoint polynomial.
	\end{itemize}
	Then the following  map is measurable:
	$$ (X^N,Z^{NM}) \mapsto \tau_N\otimes\tau_M\left(f\left(P(X^N_t\otimes I_M,Z^{NM},{Z^{NM}}^*)\right)\right) .$$
\end{prop}

\begin{proof}
	Let $(f_k)_k$ be a sequence of polynomial such that $ sup_{[-k,k]} |f-f_k| < 1/k$, then the map 
	$$ (X^N,Z^{NM}) \mapsto \tau_N\otimes\tau_M\left(f_k\left(P(X^N_t\otimes I_M,Z^{NM},{Z^{NM}}^*)\right)\right) $$
	is measurable since it is a polynomial in the entries of the matrices. Besides it converges pointwise towards
	$$ (X^N,Z^{NM}) \mapsto \tau_N\otimes\tau_M\left(f\left(P(X^N_t\otimes I_M,Z^{NM},{Z^{NM}}^*)\right)\right). $$
	Hence this map is measurable.
\end{proof}

Actually we could easily prove that this map is continuous, however we do not need it. The only reason we need this property is to justify that 
if $X^N$ is a vector of $d$ independent GUE matrices, then the random variable 
$\tau_N\otimes\tau_M\left(f\left(P(X^N_t\otimes I_M,Z^{NM},{Z^{NM}}^*)\right)\right)$ is well-defined.

\subsection{GUE random matrices}

We conclude this section by reminding the definition of Gaussian random matrices and stating a few useful properties about them. 

\begin{defi}
	\label{GUEdef}
	A GUE random matrix $X^N$ of size $N$ is a self adjoint matrix whose coefficients are random variables with the following laws:
	\begin{itemize}
		\item For $1\leq i\leq N$, the random variables $\sqrt{N} X^N_{i,i}$ are independent centered Gaussian random variables of 
		variance $1$.
		\item For $1\leq i<j\leq N$, the random variables $\sqrt{2N}\ \Re{X^N_{i,j}}$ and $\sqrt{2N}\ \Im{X^N_{i,j}}$ are independent 
		centered Gaussian random variables of variance $1$, independent of $\left(X^N_{i,i}\right)_i$.
	\end{itemize}
\end{defi}

We now present two of the most useful tools when it comes to computation with Gaussian variable, the Poincar\'e inequality and Gaussian 
integration by part. Firstly, the Poincar\'e inequality:

\begin{prop}
	\label{Poinc}
	Let $(x_1,\dots,x_n)$ be i.i.d. centered Gaussian random variable with variance $1$, let $f:\R^n\to \R$ be $\mathcal{C}^1$, then
	
	$$ \var\big(f(x_1,\dots,x_n)\big) \leq \E\big[ \norm{\nabla f (x_1,\dots,x_n)}_2^2 \big] \.$$
	
\end{prop}
For more details about the Poincar\'e inequality, we refer to Definition 4.4.2 in \cite{alice}. As for Gaussian integration by part, it comes 
from the following formula, if $Z$ is a centered Gaussian variable with variance $1$ and $f$ a $\mathcal{C}^1$ function, then

\begin{equation}
\label{IPPG}
	\E[Z f(Z)] = \E[\partial_Z f(Z)] \ .
\end{equation}

\noindent A direct consequence of this, is that if $x$ and $y$ are centered Gaussian variable with variance $1$, and 
$Z = \frac{x+\i y}{\sqrt{2}}$, then

\begin{equation}
\label{IPPG2}
	\E[Z f(x,y)] = \E[\partial_Z f(x,y)]\quad \text{ and  }\quad \E[\overline{Z} f(x,y)] = \E[\partial_{\overline{Z}} f(x,y)] \ ,
\end{equation}

\noindent where $\partial_Z = \frac{1}{2} (\partial_x + \i \partial_y)$ and $\partial_{\overline{Z}} = \frac{1}{2} (\partial_x - \i \partial_y)$. When 
working with GUE matrices, an important consequence of this are the so-called Schwinger-Dyson equation, which we summarize 
in the following 
proposition. For more information about these equations and their applications, we refer to \cite{alice}, Lemma 5.4.7.

\begin{prop}
\label{SD}
	Let $X^N$ be GUE matrices of size $N$, $Q\in \A_{d,q}$, then for any $i$,
	
	$$ \E\left[ \frac{1}{N}\tr_N(X^N_i\ Q(X^N)) \right] = \E\left[ \left(\frac{1}{N}\tr_N\right)^{\otimes 2} (\partial_i Q(X^N)) \right] . $$
	
\end{prop}

\begin{proof}
One can write $X^N_i = \frac{1}{\sqrt{N}} (X_{r,s}^i)_{1\leq r,s\leq N}$ and thus
	
\begin{align*}
	\E\left[ \frac{1}{N}\tr_N(X^N_i\ Q(X^N)) \right] &= \frac{1}{N^{3/2}} \sum_{r,s} \E\left[ X_{r,s}^i\ \tr_N(E_{r,s}\ Q(X^N)) \right] \nonumber \\
	&= \frac{1}{N^{3/2}} \sum_{r,s} \E\left[ \tr_N(E_{r,s}\ \partial_{x_{r,s}^i} Q(X^N)) \right] \\
	&= \frac{1}{N^{2}} \sum_{r,s} \E\left[ \tr_N(E_{r,s}\ \partial_i Q(X^N) \# E_{s,r}) \right] \nonumber \\
	&= \E\left[ \left(\frac{1}{N}\tr_N\right)^{\otimes 2} (\partial_i Q(X^N)) \right] . \nonumber
\end{align*}

\end{proof}

Now to finish this section we state a property that we use several times in this paper:

\begin{prop}
	\label{bornenorme}
	There exist constants $C,D$ and $\alpha$ such that for any $N\in\N$, if $X^N$ is a GUE random matrix of size $N$, then for any 
	$u\geq 0$,
	
	$$ \P\left(\norm{X^N}\geq u+D \right) \leq e^{-\alpha u N} . $$
	
	\noindent Consequently, for any $k\leq \alpha N /2 $,
	
	$$ \E\left[\norm{X^N}^k\right] \leq C^k .$$
	
\end{prop}

\begin{proof}
	The first part is a direct consequence of Lemma 2.2 from \cite{G-MS} in the specific case of the GUE. As for the second part, if 
	$k\leq \alpha N /2$, then we have, 
	\begin{align*}
	\E\left[\norm{X^N}^k\right] &= k \int_0^{\infty} \P\left(\norm{X^N}\geq u\right) u^{k-1} du \\
	&\leq k D^k + k\int_D^{\infty} e^{-N \alpha (u-D)} u^{k-1} du \\
	&\leq k D^k + k e^{DN\alpha} \int_D^{\infty} e^{(k-N \alpha)u} du \\
	&\leq k D^k + \frac{2k}{\alpha N} e^{kD}\leq C^k
	\end{align*}
for some $C$ independent of $N$ and $k$. In the third line we used  that $\ln|u|\le u$ for all positive real numbers, 	
	
\end{proof}

\section{Proof of the main result Theorem \ref{imp1}}
\label{mainsec}

\subsection{Overview of the proof}
\label{overview}

Given two families of non-commutative random variables, $(X^N\otimes I_M, Z^{NM})$ and $(x\otimes I_M, Z^{NM})$, and we want to 
study the difference between their distributions. As mentioned in the introduction, the main idea of the proof is to interpolate these two families 
with the help of a free Ornstein-Uhlenbeck process $X^{t,N} =(X^{t,N}_i)_i$ started in deterministic matrices $(X^{N,0}_i)_i$ of size $N$. 
However, as we shall explain in this subsection, we are only interested into the law of the marginals at time $t$ of this process, hence we do not 
need to define it globally. We refer to \cite{BianeSpeicher2001} for more information about it. 
Some properties of this process are well understood. For example, like in the classical case, we know its distribution at time $t$. In the classical case, if $(S_t)_t$ was an Ornstein-Uhlenbeck 
process, then it is well-known that for any function $f$ and $t\geq 0$, 

$$ \E[f(S_t)] = \E[f(e^{-t/2} S_0 + (1-e^{-t})^{1/2} X)] $$

\noindent where $X$ is a centered Gaussian random variable of variance $1$ independent of $S_0$. Likewise, if $\mu$ is the trace on the 
$\mathcal{C}^*$-algebra which contains $(X^{t,N})_{t\geq 0}$, we have for any function $f$ such that this is well-defined and $t\geq 0$,
\begin{equation}
	\mu(f(X^{t,N})) = \tau_N\left(f(e^{-t/2} X^{N,0} + (1-e^{-t})^{1/2} x)\right)
\end{equation}

\noindent where $x$ is a system of free semicircular variables, free from $\M_N(\C)$. Thus a free Ornstein-Uhlenbeck process started at time 
$t$ has the same distribution in the sense of Definition \ref{freeprob} as the family

$$ e^{-t/2} X^{N,0} + (1-e^{-t})^{1/2} x \ . $$

\noindent Consequently, from now on, we write $X^{t,N} = e^{-t/2} X^{N,0} + (1-e^{-t})^{1/2} x$. 
Since our aim in this subsection is not to give a rigorous proof but to outline the strategy used in subsection \ref{technical}, we also assume 
that we have no matrix $Z^{NM}$ and that $M=1$. Now under the assumption that this is well-defined, if 
$Q\in \A_{d,0} = \C\langle X_1,\dots,X_d\rangle$,

$$	\E\left[\frac{1}{N}\tr_{N}\Big(Q\left(X^N\right)\Big)\right] - \tau\Big(Q\left(x\right)\Big) = - \int_{0}^{\infty} \E\left[\frac{d}{dt}\Big(\tau_N\left(Q(X^{t,N})\right)\Big)\right]\ dt\ .$$
On the other hand, using the free Markov property of the free Brownian motion, we have for $Q\in \mathcal A_{d,0}$

$$ \frac{d}{dt} \tau_N(Q(X^{t,N})) = - \frac{1}{2} \sum_i\left\{ \tau_{N}\left( X_i^{t,N}  (D_iQ)(X^{t,N})\right) - \tau_N\otimes\tau_N\left( (\partial_i D_iQ) (X^{t,N}) \right)\right\} .$$

\noindent One can already recognize the Schwinger-Dyson equation, indeed thanks to Proposition \ref{SD}, one can see that

$$ \E\left[\frac{d}{dt} \tau_N(Q(X^{t,N})) \right]\bigg|_{t=0} = - \frac{1}{2} \sum_i \E\left[\tau_{N}\left( X_i^N\  (D_iQ)(X^N)\right) - \tau_N\otimes\tau_N\left( (\partial_i D_iQ) (X^N) \right) \right] = 0 \ .$$

\noindent And then, thanks to Proposition \ref{SDE},

$$ \E\left[\frac{d}{dt} \tau_N(Q(X^{t,N})) \right] \bigg|_{t=\infty}= - \frac{1}{2} \sum_i \left\{\tau\left( x_i\ (D_iQ)(x)\right) - \tau\otimes\tau\left( (\partial_i D_iQ) (x) \right)\right\} = 0 \ .$$

\noindent However what happens at time $t$ is much harder to estimate and is the core of the proof. The main idea to deal with this issue is 
to view the family $(X^N,x)$ as the asymptotic limit when $k$ goes to infinity of the family $(X^N\otimes I_k, R^{kN})$ where $R^{kN}$ are 
independent GUE matrices of size $kN$ and independent of $X^N$.

Another issue is that to prove Theorem \ref{imp1}, we would like to set $Q = f(P)$ but since $f$ is not polynomial this means that we need to 
extend the definition of operators such as $\partial_i$. In order to do so we assume that there exist $\mu$ a measure on $\R$ such that,

$$ \forall x\in\R,\quad f(x) = \int_{\R} e^{\i xy}\ d\mu(y) \ .$$

\noindent While we have to assume that the support of $\mu$ is indeed on the real line, $\mu$ can be a complex measure. However we will 
usually work with measure such that $|\mu|(\R)$ is finite. Indeed under this assumption we can use Fubini's Theorem, and we get

$$	\E\left[\frac{1}{M}\tr_{N}\Big(f\left(P(X^N)\right)\Big)\right] - \tau\Big(f\left(P(x)\right)\Big) = \int_{\R}	\left\{ \E\left[\frac{1}{N}\tr_{N}\Big(e^{\i y P\left(X^N\right)}\Big)\right] - \tau\Big(e^{\i y P\left(x\right)}\Big)\right\}\ d\mu(y)\ . $$

\noindent We can then set $Q = e^{\i y P}$. And even though this is not a polynomial function, since it is a power series, most of the 
properties associated to polynomials remain true with some assumption on the convergence. The main difficulty with this method is that we 
need to find a bound which does not depend on too high moments of $y$. Indeed  terms of the form
$$ \int_{\R}|y|^l\ d|\mu|(y) $$
 appear in our estimates. Thanks to Fourier integration we can relate the exponent $l$ to the regularity of the function $f$, thus we want to 
 find a bound with $l$ as small as possible. It turns out that with our proof $l = 4$.

\subsection{Proof of Theorem \ref{imp1}}

\label{technical}

In this section we focus on proving Theorem \ref{imp1} from which we deduce all of the important corollaries, it will be a consequence of the 
Theorem below:

\begin{theorem}
\label{imp2}
Let the following objects be given,
\begin{itemize}
	\item $X^N = (X_1^N,\dots,X_d^N)$ independent $GUE$ matrices of size $N$,
	\item $x = (x_1,\dots,x_d)$ a system of free semicircular variables,
	\item $Z^{NM} = (Z_1^{NM},\dots,Z_q^{NM})$ deterministic matrices,
	\item $P\in \A_{d,q}$ a polynomial that we assume to be self-adjoint,
	\item $f:\R\mapsto\R$ such that there exists a measure on the real line $\mu$ with $\int (1+y^4)\  d|\mu|(y)\ < +\infty$ and for any $x\in\R$,
			$$ f(x) = \int_{\R} e^{\i x y}\ d\mu(y) \ . $$
\end{itemize}

\noindent Then, there exists a polynomial $L_P$ which only depends on $P$ such that for any $N,M$,

\begin{align*}
	\Bigg| &\E\left[\frac{1}{MN}\tr_{MN}\Big(f\left(P\left(X^N\otimes I_M,Z^{NM},{Z^{NM}}^*\right)\right)\Big)\right] - \tau_N\otimes\tau_M\Big(f\left(P\left(x\otimes I_M,Z^{NM},{Z^{NM}}^*\right)\right)\Big) \Bigg| \\
	&\leq \frac{M^2}{N^2} L_P\left(\norm{Z^{NM}}\right) \int_{\R} (|y|+y^4)\  d|\mu|(y) \ .
\end{align*}

\end{theorem}

The proof is a direct corollary of Lemmas \ref{fourier} and \ref{core} below. The first one shows that the crux of the proof lies in understanding 
the following quantity:

\begin{defi}\label{defS}
	Let the following objects be given,
	\begin{itemize}
		\item $E_{r,s}$ the matrix whose every coefficient is $0$ but the one in position $(r,s)$ which is $1$,
		\item $\alpha,\beta\in [0,1]$,
		\item $A,B,C,D\in \A_{d,q}$ monomials,
		\item $ X_t^N = e^{-t/2} X^N + (1-e^{-t})^{1/2} x $
		\item $Z_t^N = (X_t^N\otimes I_M, Z^{NM},{Z^{NM}}^*)$,
		\item $S_t = (A e^{\i \beta y P} B)(Z_t^N) $,
		\item $V_t = (C e^{\i \alpha y P} D)(Z_t^N) $.
	\end{itemize}
	
	\noindent Then we define:

	\begin{align*}
	\mathcal{S}_{N,t}^{\alpha,\beta}\Big(A,B,C,D\Big) =& \ \E\left[ \frac{1}{N} \sum_{1\leq s,r\leq N} \tau_N\otimes\tau_M\Big( E_{s,r}\otimes I_M\times S_t\times E_{r,s}\otimes I_M\times V_t \Big) \right] \\
	&- \E\left[ \tau_M\Big( (\tau_N\otimes I_M)(S_t)\ (\tau_N\otimes I_M)(V_t) \Big) \right] .
	\end{align*}	
\end{defi}

We can now state the next lemma which explains why this object appears:

\begin{lemma}
	\label{fourier}
	
	Let $f$ be a function such that there exists a measure $\mu$ such that for any $x\in \R$,
	
	$$ f(x) = \int_{\R} e^{ixy} d\mu(y) $$
	
	\noindent We also assume that $\int_{\R} (1 + y^4) d|\mu|(y) < \infty $. Then one can write
	
	$$ \E\left[\frac{1}{MN}\tr\Big(f\left(P\left(X^N\otimes I_M,Z^{NM},{Z^{NM}}^*\right)\right)\Big)\right] - \tau_N\otimes\tau_M\Big(f\left(P\left(x\otimes I_M,Z^{NM},{Z^{NM}}^*\right)\right)\Big) $$
	
	\noindent as a finite linear combination of terms of the following kinds :
	
	\begin{align}\label{qa}
	\int_{0}^{\infty} e^{-t}\ \int y^2\ \int_0^1 \S^{\alpha,1-\alpha}(A,B,C,D) d\alpha\ d\mu(y)\ dt \ ,
	\end{align}
	
	\noindent and 
	\begin{align}\label{qb}
	\int_{0}^{\infty} e^{-t}\ \int y\ \S^{1,0}(A,B,C,D) d\mu(y)\ dt \ 
	\end{align}
	where the monomials $A,B,C,D\in \A_{d,q}$  and the coefficients of the linear combination are uniquely determined by $P$.
\end{lemma}

\begin{proof}
	
	First, we define the natural interpolation between the trace of matrices at size $N$ and the trace of semicircular variables,
	
	\begin{equation*}
	s(t,y) = \E\left[ \tau_N\otimes\tau_M\left(e^{\i y P(Z_t^N)}\right) \right]\ .
	\end{equation*}
	
	\noindent By definition of $f$ we have
	
	\begin{eqnarray*} \int_{\R} s(0,y)\  d\mu(y) &=& \E\left[\frac{1}{MN}\tr_{MN}\Big(f(P(X^N\otimes I_M,Z^{NM},{Z^{NM}}^*))\Big)\right] ,\\
	 \int_{\R} s(\infty,y)\  d\mu(y) &=& \tau_N\otimes\tau_M\Big(f(P(x\otimes I_M,Z^{NM},{Z^{NM}}^*))\Big) .\end{eqnarray*}
	
	\noindent Thus under the assumption that this is well-defined, we have
	
	\begin{align}
		\E&\left[\frac{1}{MN}\tr_{MN}\Big(f\left(P\left(X^N\otimes I_M,Z^{NM},{Z^{NM}}^*\right)\right)\Big)\right] - \tau_N\otimes\tau_M\Big(f\left(P\left(x\otimes I_M,Z^{NM},{Z^{NM}}^*\right)\right)\Big)  \nonumber \\
		&= - \int_{0}^{\infty} \int_{\R} \partial_ts(t,y)\  d\mu(y)\ dt \label{integ} \ .
	\end{align}
	
	\noindent We compute
	
	\begin{align*}
	\label{impquant}
	\partial_ts(t,y) = \i y \frac{e^{-t}}{2}\ \E\left[ \tau_N\otimes\tau_M\left(e^{\i y P(Z_t^N)}\ \sum_i \partial_i P(Z_t^N) \# \left( \left( \frac{x_i}{(1-e^{-t})^{1/2}} - e^{t/2} X_i^N\right) \otimes I_M \right) \right) \right] .
	\end{align*}
	
	\noindent Since we assumed  that $\mu$ is such that $\int (1+y^4) d\mu(y) < +\infty $ and that since $X_i^N$ and $x_i$ have all 
	moments uniformly bounded by Lemma \ref{bornenorme}, we can find a constant $C$ independent from $y$ and $t$ such that
	
	$$ |\partial_ts(t,y)| \leq C\ y e^{-t/2} \, $$	
	
	\noindent we can deduce that \eqref{integ} is well-defined. Besides, writing $P=\sum c_Q(P) Q$ with monomials $Q\in \mathcal A_{d,q}$, we get
	\begin{equation}\label{exps}\partial_ts(t,y)= \i y \frac{e^{-t}}{2}\sum c_Q(P)\sum_{Q= B X_i A}
	 \E\left[ \tau_N\otimes\tau_M\left(e^{\i y P(Z_t^N)}\ B(Z_t^N) \left( \frac{x_i}{(1-e^{-t})^{1/2}} - e^{t/2} X_i^N\right) \otimes I_M A(Z_t^N)\right) \right].\end{equation}
	Hence, $\partial_t s$ is a finite linear combination of terms of the form
	\begin{equation}ye^{-t} \label{dif}S_t(A,B)= ye^{-t} S_t^1(A,B)- ye^{-t} S_t^2(A,B)\end{equation}
	with $$S^1_t(A,B)=S_t(A,B, (1-e^{-t})^{-1/2}x_i)\mbox{ and }S^2_t(A,B)=S_t(A,B, e^{t/2} X_i^N)$$
	where
	\begin{equation}
	\label{prob}
	S_t(A,B,G)= \E\left[ \tau_N\otimes\tau_M\left( A(Z_t^N)\  e^{\i y P(Z_t^N)} B(Z_t^N) G\otimes I_M \right) \right].
	\end{equation}
	
	\noindent  We first study $S^2_t(A,B)$. We denote by $Q= A e^{\i y P} B$. 
	We want to use Gaussian integration by part : if we set $\sqrt{N} X_i^N = (x_{s,r}^i)_{1\leq s,r\leq N}$, then with $\partial_{x_{s,r}^i}$ as 
	in equations \eqref{IPPG} and \eqref{IPPG2}, thanks to Duhamel formula
	
	\begin{align}
	\sqrt{N} e^{t/2}\ \partial_{x_{s,r}^i} Q(Z_t^N) =& \  \partial_i A(Z_t^N) \# (E_{r,s}\otimes I_M)\ e^{\i y P(Z_t^N)} B(Z_t^N)\nonumber  \\
	&+ \i y \int_0^1 A(Z_t^N) e^{\i (1-\alpha) y P(Z_t^N)}\ \partial_i P(Z_t^N) \# (E_{r,s}\otimes I_M)\ e^{\i \alpha y P(Z_t^N)} B(Z_t^N)\  d\alpha\label{deri} \\
	&+ A(Z_t^N) e^{\i y P(Z_t^N)}\ \partial_i B(Z_t^N) \# (E_{r,s}\otimes I_M) \.\nonumber
	\end{align}
	
	\noindent Consequently, expanding in $S^2_t(A,B)$ the product by $X_i^N$ in terms of its entries, we have	
	
	\begin{align}
	\label{premierterm}
	S^2_t(A,B)=&e^{t/2} \E\left[ \tau_N\otimes\tau_M\left( ( A e^{\i y P} B)(Z_t^N)\ X_i^N\otimes I_M \right) \right] \nonumber \\
	=& \ N^{-1/2} e^{t/2} \sum_{1\leq s,r\leq N} \E\left[ x^i_{s,r}\ \tau_N\otimes\tau_M\Big( E_{s,r}\otimes I_M\ ( A e^{\i y P} B)(Z_t^N) \Big) \right] \nonumber \\
	=& \ \frac{1}{N} \sum_{1\leq s,r\leq N} \E\left[ \tau_N\otimes\tau_M\Big( E_{s,r}\otimes I_M\ e^{t/2} \partial_{x_{s,r}^i} Q(Z_t^N) \Big) \right] \nonumber \\
	=& \ \E\left[ \frac{1}{N} \sum_{1\leq s,r\leq N} \tau_N\otimes\tau_M\Big( E_{s,r}\otimes I_M\ \partial_i A \# (E_{r,s}\otimes I_M)\ e^{\i y P} B \Big) \right] \nonumber \\
	&+ \i y \int_0^1 \E\left[ \frac{1}{N} \sum_{1\leq s,r\leq N} \tau_N\otimes\tau_M\Big( E_{s,r}\otimes I_M\ A e^{\i (1-\alpha) y P}\ \partial_i P \# (E_{r,s}\otimes I_M)\ e^{\i \alpha y P} B \Big) \right] d\alpha \nonumber \\
	&+ \E\left[ \frac{1}{N} \sum_{1\leq s,r\leq N} \tau_N\otimes\tau_M\Big( E_{s,r}\otimes I_M\ A e^{\i y P}\ \partial_i B \# (E_{r,s}\otimes I_M) \Big) \right] 
	\end{align}
	where $A,B,P$ are evaluated at $Z^N_t$.
	\noindent To deal with $S^1_t(A,B)$, since a priori we defined free integration by parts only for polynomials, we expand the exponential 
	as a power series,
	
	\begin{align*}
	&\tau_N\otimes\tau_M\left( A(Z_t^N)\  e^{\i y P(Z_t^N)}\ B(Z_t^N) \frac{x_i \otimes I_M}{(1-e^{-t})^{1/2}} \right) \\
	&= \sum_{k\geq 0} \frac{1}{k!}\ \tau_{N}\otimes\tau_M \left( A(Z_t^N)\  (\i y P(Z_t^N))^k\ B(Z_t^N) \frac{x_i \otimes I_M}{(1-e^{-t})^{1/2}} \right) .
	\end{align*}
	
	\noindent We define $(\tau_N\otimes I_M) \bigotimes (\tau_N\otimes I_M) : (\A_N\otimes \M_M(\C))^{\otimes 2} \to M_M(\C)$ the linear application which is defined  on simple tensor by $(\tau_N\otimes I_M) \bigotimes (\tau_N\otimes I_M)(A\otimes B) = (\tau_N\otimes I_M)(A) \times (\tau_N\otimes I_M)(B)$. Hence, thanks to Proposition \ref{SDE}, with the convention that $ A\times (B\otimes C)\times D = (AB)\otimes (CD) $, we have 
	
	\begin{align*}
		\tau_{N}&\otimes\tau_M \left( A(Z_t^N)\  (\i y P(Z_t^N))^k\ B(Z_t^N) \frac{x_i \otimes I_M}{(1-e^{-t})^{1/2}} \right) \\
		=&\ \tau_M\left( (\tau_N\otimes I_M) \bigotimes (\tau_N\otimes I_M) \Big( \partial_i A(Z_t^N)\  (\i y P(Z_t^N))^k\ B(Z_t^N) \Big) \right) \\
		&+ \i y \tau_M\left( (\tau_N\otimes I_M) \bigotimes (\tau_N\otimes I_M) \Big( A(Z_t^N)\  (\i y)^{k-1} \sum_{1\leq l\leq k}P(Z_t^N)^{l-1} \partial_i P(Z_t^N)\ P(Z_t^N)^{k-l} \ B(Z_t^N) \Big) \right) \\
		&+ \tau_M\left( (\tau_N\otimes I_M) \bigotimes (\tau_N\otimes I_M) \Big( A(Z_t^N)\  (\i y P(Z_t^N))^k\ \partial_i B(Z_t^N) \Big) \right) . \\
	\end{align*}
	
	\noindent Now we can use the fact that 
	
	$$ \frac{1}{k!} = \int_0^1 \frac{ \alpha^{l-1} (1-\alpha)^{k-l}}{(l-1)! (k-l)!} d\alpha \,$$
	
	\noindent to deduce that
	
	\begin{align*}
		&\tau_M\left( (\tau_N\otimes I_M) \bigotimes (\tau_N\otimes I_M) \Big( A(Z_t^N)\ \sum_{k\geq 1} \frac{(\i y)^{k-1}}{k!} \sum_{l=1}^k P(Z_t^N)^{l-1} \partial_i P(Z_t^N)\ P(Z_t^N)^{k-l} B(Z_t^N) \Big) \right) \\
		&= \int_0^1 \sum_{k\geq 1} \sum_{l=1}^k \tau_M\Bigg( (\tau_N\otimes I_M) \bigotimes (\tau_N\otimes I_M) \Bigg( A(Z_t^N)\  \frac{(\i y \alpha P(Z_t^N))^{l-1}}{(l-1)!}\ \partial_i P(Z_t^N) \\
		&\quad\quad\quad\quad\quad\quad\quad\quad\quad\quad\quad\quad\quad\quad\quad\quad\quad\quad\quad\quad\quad \frac{(\i y(1-\alpha)P(Z_t^N))^{k-l}}{(k-l)!} B(Z_t^N) \Bigg) \Bigg)d\alpha \\
		&= \int_0^1  \tau_M\left( (\tau_N\otimes I_M) \bigotimes (\tau_N\otimes I_M) \Big( A(Z_t^N)\ e^{\i (1-\alpha) y P(Z_t^N)}\ \partial_i P(Z_t^N)\ e^{\i \alpha y P(Z_t^N)} \ B(Z_t^N) \Big) \right) d\alpha \. \\
	\end{align*}

	\noindent And thus, by summing, we obtain
	
	\begin{align*}
	S^1_t(A,B)=& \  \tau_M\Big( (\tau_{N}\otimes I_M) \bigotimes (\tau_{N}\otimes I_M) \Big( \partial_i A\ e^{\i y P} B\Big) \Big) \\
	&+ \i y \int_0^1 \tau_M\Big( (\tau_{N}\otimes I_M) \bigotimes (\tau_{N}\otimes I_M) \Big( A e^{\i (1-\alpha) y P}\ \partial_i P\ e^{\i \alpha y P} B \Big) \Big) d\alpha \\
	&+ \tau_M\Big( (\tau_{N}\otimes I_M) \bigotimes (\tau_{N}\otimes I_M) \Big( A\ e^{\i y P}\ \partial_i B\Big) \Big) \.
	\end{align*}
	
	\noindent Therefore, after making the difference \eqref{dif} to compute $S_t(A,B)$, we conclude that the difference we wish to estimate  in \eqref{integ} is a linear combination of terms,  whose coefficients only depend on $P$, of the  form \eqref{qa} and \eqref{qb}.

\end{proof}

We need to study the quantity $\S^{\alpha,\beta}(A,B,C,D)$. Let us first explain why one can expect it to be small. 
Let $(g_i)_{1\leq i\leq N}$ be the canonical basis of $\C^N$ so that $E_{r,s}= g_r g_s^*$.   We observe that 
 $\mathcal{S}_{N,0}^{\alpha,\beta}(A,B,C,D) = 0$, since

\begin{align*}
&\frac{1}{N} \sum_{1\leq s,r\leq N} \tau_N\otimes\tau_M\Big( E_{s,r}\otimes I_M\ S_0\ E_{r,s}\otimes I_M\ V_0 \Big) \\
&= \frac{1}{N^2} \sum_{1\leq s,r\leq N} \tau_M\Big( g_r^*\otimes I_M\ S_0\ g_r\otimes I_M\ g_s^*\otimes I_M\ V_0\ g_s\otimes I_M\ \Big) \\
&= \tau_M\Big( (\tau_N\otimes I_M)(S_0)\ (\tau_N\otimes I_M)(V_0) \Big) \.\\
\end{align*}

\noindent Let us now estimate $\mathcal{S}_{N,\infty}^{\alpha,\beta}(A,B,C,D)$. We first notice that if $X,Y\in \A_N$ are free from $\M_N(\C)$ 
then (with constants being identified with constants times identity):

\begin{itemize}
	\item If $r\neq s$: $\tau_N\Big( E_{r,s} (X-\tau_N(X)) E_{s,r} (Y-\tau(Y)) \Big) = 0 $ .
	\item If $r = s$: $\tau_N\Big( \left(E_{r,r}-\frac{1}{N}\right) (X-\tau_N(X)) \left(E_{r,r}-\frac{1}{N}\right) (Y-\tau_N(Y)) \Big) = 0 $ .
\end{itemize}

\noindent Consequently, since $\tau_N(E_{r,s}E_{s,r})=1/N$ for all $r,s$, we get:

\begin{itemize}
	\item If $r\neq s$: $\tau_N\Big( E_{r,s} X E_{s,r} Y \Big) = \frac{1}{N} \tau_N(X) \tau_N(Y) $ .
	\item If $r = s$: $\tau_N\Big( E_{r,s} X E_{s,r} Y \Big) = \frac{1}{N} \tau_N(X) \tau_N(Y) + \frac{1}{N^2} \Big( \tau_N(XY) -\tau_N(X)\tau_N(Y) \Big)$ .
\end{itemize}
Hence 
$$ \frac{1}{N} \sum_{1\leq s,r\leq N} \tau_N\Big( E_{r,s} X E_{s,r} Y \Big) = \tau_N(X) \tau_N(Y) + \frac{1}{N^2} \Big( \tau_N(XY) -\tau_N(X)\tau_N(Y) \Big).$$
 This implies that $N^2 \mathcal{S}_{N,\infty}^{\alpha,\beta}(A,B,C,D)$ is bounded by a constant independent of $N$ or $y$ since

\begin{align*}
&\frac{1}{N} \sum_{1\leq s,r\leq N} \tau_N\otimes\tau_M\Big( E_{s,r}\otimes I_M\ S_\infty \ E_{r,s}\otimes I_M\ V_\infty \Big) \\
&= \tau_M\Big( (\tau_N\otimes I_M)(S_\infty)\ (\tau_N\otimes I_M)(V) \Big) + \frac{1}{N^2} \Big( \tau_N\otimes\tau_M(S_\infty V_\infty) - \tau_M( \tau_N\otimes I_M (S) \tau_N\otimes I_M (V_\infty) ) \Big) \. \\
\end{align*}

\noindent  With this in mind, we now study what happens at time $t$. More precisely we show:

\begin{lemma}
\label{core}

There is a polynomial $L$ which only depends on $A,B,C,D$ and $P$ such that for any $\alpha,\beta\in [0,1]$, $N\in\N$, $t\in\R^+$ and $y\in\R$,

$$ \Big| \S^{\alpha,\beta}\Big(A,B,C,D\Big) \Big| \leq \frac{(1+y^2) M^2}{N^2}  L(\norm{Z^{NM}}) \.$$

\end{lemma}

This lemma is a direct consequence of Lemmas \ref{grcalc} and \ref{nondiag}, but firstly we need the following proposition. 
It justifies that the family $(X^N\otimes I_M, x\otimes I_M, Z^{NM})$ is actually the asymptotic distribution (in the sense of Definition 
\ref{freeprob}) as $k$ goes to infinity of the family $(X^N\otimes I_{kM}, R^{kN}\otimes I_M, Z^{NM}\otimes I_k)$ where $R^{kN}$ are 
independent GUE random matrices of size $kN$. The advantage of this representation is that it allows us to use classical analysis, and to 
treat the GUE variables and the semi-circle variables in a more symmetric way.  A direct proof using semi-circular variables should however 
be possible. 

\begin{prop}
	\label{limite_trace}
	
	If $R^{kN}$ are independent GUE random matrices of size $kN$, independent of $X^N$, we set
	
	\begin{equation*}
	U_t^k = \Big(\left(e^{-t/2}X^N\otimes I_k + (1-e^{-t})^{1/2} R^{kN}\right) \otimes I_M\ , Z^{NM}\otimes I_k, {Z^{NM}}^*\otimes I_k\Big) \.
	\end{equation*}
	
	\noindent Then if $q = A e^{\i \beta y P} B$, we have that $\P_{X^N}$-almost surely for any $t$,
	
	$$ (\tau_N\otimes I_M)\big(q(Z_t^N)\big) = \lim_{k\to\infty} \E_R\left[ (\tau_{kN}\otimes I_M)\big(q(U_t^k)\big) \right] \,$$
	
	\noindent where $\E_R$ is the expectation with respect to $R^{kN}$. Observe here that $M,N$ are kept fixed.
	
\end{prop}

\begin{proof}
	
	This proposition is mostly a corollary of Theorem 5.4.5 of \cite{alice}, indeed this theorem states that if $R^{kN}$ are GUE matrices and 
	$D^{kN}$ are deterministic matrices such that
	
	$$ \sup_{l\in\N}\ \max_{i}\ \sup_{k\in\N} \left( \frac{1}{N} \tr(|D_i^{kN}|^l) \right)^{1/l} < \infty \. $$
	
	\noindent Then, if $D^{kN}$ converges in distribution towards a family of non-commutative random variables $d$, the family 
	$(R^{kN},D^{kN})$ in the non-commutative probability space $(\M_{kN}(\C),*,\E[\frac{1}{kN}\tr])$ converges in distribution towards the 
	family $(x,d)$ where $x$ is a system of free semicircular variables free from $d$. In our situation we can write for every $i$,
	
	$$ Z_i^{NM} = \sum_{1\leq r,s\leq N} E_{r,s}\otimes A_{r,s,i}^M \.$$
	
	\noindent Thus, if $E^N = (E_{r,s})_{1\leq r,s\leq N}$, we fix $D^{k,N} = (X^N\otimes I_k, E^N\otimes I_k)$, and we can apply Theorem 5.4.5 from \cite{alice} to get that for any non-commutative polynomial $P$,
	
	$$ \lim_{k\to\infty} \E_R\left[\tau_{kN}(P(R^{kN},D^{k,N})) \right] = \tau_N\left(P(x,D^{k,1})\right) .$$
	
	\noindent Consequently, for any non-commutative polynomial $P$, we also have
	
	$$ \lim_{k\to\infty} \E_R\left[\tau_{kN}\otimes I_M (P(R^{kN},D^{k,N},A^M,{A^M}^*)) \right] = \tau_N\otimes I_M  \left(P(x,X^N,E^N,A^M,{A^M}^*)\right) .$$
	
	\noindent Hence, for any $P\in \A_{d,q}$,
	
	\begin{equation}\label{asf} \lim_{k\to\infty} \E_R\left[\tau_{kN}\otimes I_M (P(U_t^k)) \right] = \tau_N\otimes I_M \left(P(Z_t^N)\right) .\end{equation}
	
	Thanks to Property \ref{bornenorme}, we know that there exist $\alpha>0$ and $D<\infty$ such that for all $u\geq D$, for $N$ large 
	enough,
	
	\begin{equation}
		\label{LDP}
		\P\left(\norm{R_1^{kN}}\geq u\right) \leq e^{-\alpha\ u\ kN} \.
	\end{equation}
	
	\noindent Since if $c_M(P)$ is the coefficient of $P$ associated with the monomial $M$, one has
	
	$$ \norm{P(U_t^k)} \leq \sum_{M \mbox{ monomials}} |c_M(P)| \norm{M(U_t^k)} \,$$
	
	\noindent there exist constants $L$ and $C$ which do depend on $\norm{Z^{NM}_j}$ and $\norm{X_i^N}$ such that for $N$ large 
	enough
	
	\begin{equation}
		\label{LDP2}
		\P\left(\norm{P(U_t^k)}\geq C\right) \leq e^{-L kN} \.
	\end{equation}
	
	\noindent Knowing this, let $f_{\varepsilon}\in \C[X]$ be a polynomial which is $\varepsilon$-close from $x\mapsto e^{\i \beta y x}$ on the 
	interval $[-1-C,C+1]$. Since one can always assume that $C> \norm{P(Z_t^N)}$, we have, with $q = A e^{\i \beta y P} B$ :
	
	$$ \| (\tau_N\otimes I_M)\big(q(Z_t^N)\big) - (\tau_{N}\otimes I_M)\big((Af_{\varepsilon}(P)B)(Z_t^N)\big) \| \leq  D\varepsilon \,$$
	where $D$ is some constant which can depend on the dimensions $N,M$ but not on $k$.
	
	\noindent Thus
	
	\begin{align*}
		&\| (\tau_N\otimes I_M)\big(q(Z_t^N)\big) - \E_R\left[ (\tau_{kN}\otimes I_M)\big(q(U_t^k)\big) \right] \|\le D \varepsilon + D \E_R\left[ \norm{(q - Af_{\varepsilon}(P)B)(U_t^k)} \1_{\norm{P(U_t^k)}\geq C+1} \right] \\
		&\qquad\qquad\qquad\qquad  +\| (\tau_N\otimes I_M)\big((Af_{\varepsilon}(P)B)(Z_t^N)\big) - \E_R\left[ (\tau_{kN}\otimes I_M)\big((Af_{\varepsilon}(P)B)(U_t^k)\big) \right] \|  \\
	\end{align*}
	
	\noindent The last term goes to zero as $k$ goes to infinity by \eqref{asf}. Besides
	
	\begin{align*}
		&\E_R\left[ \norm{(q - Af_{\varepsilon}(P)B)(U_t^k)} \1_{\norm{P(U_t^k)}\geq C+1} \right] \\
		&\qquad\qquad \leq \E_R\left[ \left(\norm{A(U_t^k)} \norm{B(U_t^k)} + \norm{(Af_{\varepsilon}(P)B)(U_t^k)}\right)^2 \right]^{1/2} \P( \norm{P(U_t^k)}\geq C+1)^{1/2} \.
	\end{align*}
	
	\noindent The first term is bounded independently of $k$ thanks to \eqref{LDP} and the second converges exponentially fast towards 
	$0$ thanks to \eqref{LDP2}. Consequently
	
	$$\limsup\limits_{k\to \infty} \| (\tau_N\otimes I_M)\big(q(Z_t^N)\big) - \E_R\left[ (\tau_{kN}\otimes I_M)\big(q(U_t^k)\big) \right] \| \leq D\varepsilon \.$$
	
	\noindent Hence the conclusion since the left hand side does not depend on $\varepsilon$.
	
\end{proof}

Recall that by definition

\begin{equation}\label{defL}
\mathcal{S}_{N,t}^{\alpha,\beta}\Big(A,B,C,D\Big):=\mathbb E[\Lambda_{N,t}^{\alpha,\beta}\Big(A,B,C,D\Big)]\end{equation}
with, following the notations of Definition \ref{defS} :
\begin{align*}
	\Lambda_{N,t}^{\alpha,\beta}\Big(A,B,C,D\Big) =& \ \frac{1}{N} \sum_{1\leq s,r\leq N} \tau_N\otimes\tau_M\Big( E_{s,r}\otimes I_M\times S_t\times E_{r,s}\otimes I_M\times V_t \Big)  \\
	&- \tau_M\Big( (\tau_N\otimes I_M)(S_t)\ (\tau_N\otimes I_M)(V_t) \Big)  .
	\end{align*}
By Proposition \ref{limite_trace}, we deduce that
\begin{equation}\label{defLa}
\Lambda_{N,t}^{\alpha,\beta}\Big(A,B,C,D\Big) =\lim_{k\rightarrow\infty} \Lambda_{k,N,t}^{\alpha,\beta}\Big(A,B,C,D\Big)\end{equation}
where $\Lambda_{k,N,t}^{\alpha,\beta}\Big(A,B,C,D\Big)$ equals
\begin{eqnarray}\label{major}
&&\E_R\Big[ \frac{1}{N} \sum_{1\leq s,r\leq N} \tau_{kN}\otimes\tau_M\Big( E_{s,r}\otimes I_k\otimes I_M\ (A e^{\i \beta y P} B)(U_t^k)\ E_{r,s}\otimes I_k\otimes I_M\ (C e^{\i \alpha y P} D)(U_t^k)  \Big) \Big] \nonumber\\
&&\qquad - \tau_M\big( \E_R\left[\tau_{kN}\otimes I_M (A e^{\i \beta y P} B)(U_t^k)\right] \E_R\left[\tau_{kN}\otimes I_M(C e^{\i \alpha y P} D)(U_t^k) ) \right] \big)\end{eqnarray}

We can now prove the following intermediary lemma in view of deriving Lemma \ref{core}.

\begin{lemma}
	\label{grcalc}
	We define $U_t^k$ as in Proposition \ref{limite_trace}, we set
	\begin{itemize}
		\item $P_{1,2} = I_N\otimes E_{1,2}\otimes I_M$,
		\item $Q = (A e^{\i \beta y P} B)(U_t^k)$,
		\item $ T = (C e^{\i \alpha y P} D)(U_t^k) $.
	\end{itemize}
	
	\noindent Then there is a constant $C$ and a polynomial $L$ which only depend on $A,B,C,D$ and $P$ such that for any $\alpha,\beta\in [0,1]$, $M,N\in\N$, $t\in\R^+$ and $y\in\R$,
	
	\begin{align}
	\label{intermedi}
	|\Lambda_{k,N,t}^{\alpha,\beta}\Big(A,B,C,D\Big)|
	 \leq &\ \frac{(1+y^2) M^2}{N^2} L\left(\norm{Z^{NM}},\norm{X^N}\right) \\
		&+ k^3 \left| \tau_M\left( \E_R\left[(\tau_{kN}\otimes I_M)(QP_{1,2})\right] \E_R\left[(\tau_{kN}\otimes I_M)(TP_{1,2})\right] \right) \right|. \nonumber
	\end{align}
	
\end{lemma}

\begin{proof}
We denote in short $\Lambda_{k,N,t}^{\alpha,\beta}\Big(A,B,C,D\Big)=\Lambda_{k,N,M}=\E_R[\Lambda_{k,N,M}^1]-\Lambda_{k,N,M}^2$
with
\begin{eqnarray}
\Lambda_{k,N,M}^1&=&\frac{1}{N} \sum_{1\leq s,r\leq N} \tau_{kN}\otimes I_M\left( E_{s,r}\otimes I_k\otimes I_M\ Q\ E_{r,s}\otimes I_k\otimes I_M\ T\right)\nonumber\\
\Lambda_{k,N,M}^2&=& \tau_M\big( \E_R\left[\tau_{kN}\otimes I_M Q\right] \E_R\left[\tau_{kN}\otimes I_M(T) ) \right] \big)\label{major1}\end{eqnarray}
 Let $(g_i)_{i\in [1,N]}$ and $(f_i)_{i\in [1,k]}$ be the canonical basis of $\C^N$ and $\C^k$, $E_{i,j}$ is the matrix whose only non-zero coefficient is $(i,j)$ and this coefficient has value $1$, the size of the matrix $E_{i,j}$ will depend on the context. We use the fact that $E_{r,s} = g_r g_s^*$ and $I_k = \sum_l E_{l,l}$ with $E_{l,l} = f_l^*f_l$ to deduce that

\begin{align}
\label{alg}
\Lambda_{k,N,M}^1&= \frac{1}{N} \sum_{1\leq s,r\leq N} \sum_{1\leq l,l'\leq k} \tau_{kN}\otimes I_M\left( E_{s,r}\otimes E_{l,l}\otimes I_M\ Q\ E_{r,s}\otimes E_{l',l'}\otimes I_M\ T\right)  \nonumber \\
&= \frac{1}{N^2k} \sum_{1\leq l,l'\leq k} \sum_{1\leq r\leq N} g_r^*\otimes f_l^*\otimes I_M\ Q\ g_r\otimes f_{l'}\otimes I_M \sum_{1\leq s\leq N} g_s^*\otimes f_{l'}^*\otimes I_M\ T\ g_s\otimes f_l\otimes I_M \nonumber\\
&= \frac{1}{k} \sum_{1\leq l,l'\leq k} \left(\tau_N\otimes I_M\right)( I_N\otimes f_l^*\otimes I_M\ Q\ I_N\otimes f_{l'}\otimes I_M)\ \left(\tau_N\otimes I_M\right)(I_N\otimes f_{l'}^*\otimes I_M\ T\ I_N\otimes f_l\otimes I_M) \nonumber\\
&= k \sum_{1\leq l,l'\leq k} (\tau_{kN}\otimes I_M)\big( Q\ I_N\otimes E_{l',l}\otimes I_M\big)\ (\tau_{kN}\otimes I_M)\big( T\ I_N\otimes E_{l,l'}\otimes I_M\big).
\end{align}

\noindent The last line of the above equation prompts us to set $P_{l',l} = I_N\otimes E_{l',l}\otimes I_M$. If $(e_i)_{i\in [1,M]}$ is the canonical basis of $\C^M$, we set

$$ F^q_{l,l',u,v}(R^{kN}) = e_u^*\ (\tau_{kN}\otimes I_M)\Big( q\big((e^{-t/2}X^N\otimes I_k + (\frac{1-e^{-t}}{Nk})^{1/2} R^{kN} )\otimes I_M, Z^{NM},{Z^{NM}}^*\big)\ P_{l',l}\Big)\ e_v$$

\noindent with $q = Q=A e^{\i \beta y P} B$ or $q=T=C e^{\i \alpha y P} D$.  We thus have with \eqref{alg}

\begin{eqnarray}
	\label{CS}
	\E_R\left[\Lambda_{k,N,M}^1\right]&=&k \sum_{1\leq l,l'\leq k} \tau_M\left( \E_R\left[ (\tau_{kN}\otimes I_M)\big( Q\ P_{l',l}\big)\ (\tau_{kN}\otimes I_M)\big( T\ P_{l,l'}\big) \right] \right) \\
	&=&\ \frac{k}{M} \sum\limits_{\substack{ 1\leq l,l' \leq k \\ 1\leq u,v\leq M }} \cov_R\left(F^Q_{l,l',u,v}(R^{kN}), F^T_{l',l,u,v}(R^{kN})\right)
	 \nonumber \\
	&&+ k \sum_{1\leq l,l'\leq k} \tau_M\left( \E_R\left[ (\tau_{kN}\otimes I_M)\big( Q\ P_{l',l}\big) \right]\ \E_R\left[(\tau_{kN}\otimes I_M)\big( T\ P_{l,l'}\big) \right] \right) \nonumber \.
\end{eqnarray}

\noindent However, the law of $U_t^k$ is invariant under conjugation by $I_N\otimes U\otimes I_M$, where $U\in M_k(\C)$ is a permutation matrix. Therefore, if $l=l'$, $\E_R[\tau_{kN}( Q\ P_{l',l})] = \E_R[\tau_{kN}( Q\ P_{1,1})]$, and if $l\neq l'$, $\E_R[\tau_{kN}( Q\ P_{l',l})] = \E_R[\tau_{kN}( Q\ P_{1,2})]$. We get the same equation when replacing $Q$ by $T$. Consequently, we get

\begin{eqnarray*}
&&k \sum_{1\leq l,l' \leq k} \E_R\Big[(\tau_{kN}\otimes I_M)\big( Q\ P_{l',l}\big)\Big] \E_R\Big[(\tau_{kN}\otimes I_M)\big( T\ P_{l,l'}\big)\Big] \nonumber\\
&&\qquad=\ k^2\ \E_R[(\tau_{kN}\otimes I_M)(QP_{1,1})]\ \E_R[(\tau_{kN}\otimes I_M)(TP_{1,1})] \\
&&\qquad+ (k-1)k^2\ \E_R[(\tau_{kN}\otimes I_M)(QP_{1,2})]\ \E_R[(\tau_{kN}\otimes I_M)(TP_{1,2})] \. \nonumber
\end{eqnarray*}
where the first term in the right hand side equals $\Lambda_{k,N,M}^2=\E_R[(\tau_{kN}\otimes I_M)(Q)]\ \E_R[(\tau_{kN}\otimes I_M)(T)]$ 
because $I_M=\sum_l P_{l,l}$.
Thus equation \eqref{CS} yields 
\begin{eqnarray}
	\label{CS2}
	|\Lambda_{k,N,M}|&\le &\frac{k}{M} \sum\limits_{\substack{ 1\leq l,l' \leq k \\ 1\leq u,v\leq M }} \left|  \cov_R\left(F^Q_{l,l',u,v}(R^{kN}), F^T_{l',l,u,v}(R^{kN})\right)
 \right| \\
	&&+ \left| k^3 \tau_M\Big( \E_R[(\tau_{kN}\otimes I_M)(QP_{1,2})]\ \E_R[(\tau_{kN}\otimes I_M)(TP_{1,2})] \Big) \right| \.\nonumber
\end{eqnarray}
Hence, we only need to bound the first term to complete the proof of the lemma.
Thanks to  Cauchy-Schwartz's inequality, it is enough to bound the covariance of $F^q_{l,l',u,v}(R^{kN})$, for $q=Q$ and $T$.
 To study these variances, we shall use the Poincar\'e inequality, see Proposition \ref{Poinc}. If we set $x_{r,s}^i$ and $y_{r,s}^i$ the real and 
 imaginary part of $ \sqrt{2kN} (R_i^{kN})_{r,s}$ for $r<s$ and $x_{r,r}^i = \sqrt{kN} (R_i^{kN})_{r,r}$, then these are real centered Gaussian 
 random variables of variance $1$ and one can view $F^q_{l,l',u,v}$ as a function on $ (x_{r,s}^i)_{r\leq s,i}$ and $ (y_{r,s}^i)_{r<s,i}$. By a 
 computation 
similar to \eqref{deri}, we find

\begin{equation}
\label{gradiant}
\frac{kN}{1-e^{-t}}  \norm{\nabla F^q_{l,l',u,v}}_2^2 
= \sum_i \sum_{1\leq r,s\leq kN} e_u^*\ (\tau_{kN}\otimes I_M)\Big( \partial_i q \# E_{r,s}\otimes I_M\ P_{l',l} \Big) e_ve_v^*\ (\tau_{kN}\otimes I_M)\Big( \partial_i q \# E_{r,s}\otimes I_M\ P_{l',l} \Big)^* e_u \nonumber \\\nonumber .
\end{equation}
 It is worth noting that here the matrices $E_{r,s}$ have size $kN$ in this formula. Thanks to Poincar\'e inequality (see Proposition \ref{Poinc}), 
 we deduce 

\begin{eqnarray}
&&\frac{k}{M}\sum_{1\leq u,v\leq M} \var_R(F^q_{l,l',u,v}(R_{kN})) \leq \frac{k}{M}\sum_{1\leq u,v\leq M} \E\left[ \norm{\nabla F^q_{l,l',u,v}}_2^2 \right] \nonumber\\
&&\leq\frac{1}{N} \sum_i \sum_{1\leq r,s\leq kN} \E_R\Bigg[ \frac{1}{M}\sum_{1\leq u,v\leq M} e_u^*\ (\tau_{kN}\otimes I_M)\Big( \partial_i q \# E_{r,s}\otimes I_M\ P_{l',l} \Big) e_ve_v^*\label{comb1} \\
&&\quad \quad \quad \quad \quad \quad \quad \quad \quad \quad \quad \quad \quad \quad \quad \times (\tau_{kN}\otimes I_M)\Big( \partial_i q \# E_{r,s}\otimes I_M\ P_{l',l} \Big)^* e_u \Bigg] \nonumber \\
&&\leq \frac{1}{N} \sum_i \sum_{1\leq r,s\leq kN} \E_R\left[ \tau_M\left( (\tau_{kN}\otimes I_M)\Big( \partial_i q \# E_{r,s}\otimes I_M\ P_{l',l} \Big) (\tau_{kN}\otimes I_M)\Big( \partial_i q \# E_{r,s}\otimes I_M\ P_{l',l} \Big)^* \right) \right] \.\nonumber 
\end{eqnarray}

\noindent Moreover we have, if $e_l$ is an orthornormal basis of $\mathbb C^k$,

\begin{align}
\label{comb2}
	&\sum\limits_{1\leq l,l' \leq k} \tau_M\left( (\tau_{kN}\otimes I_M)\Big( \partial_i q \# E_{r,s}\otimes I_M\ P_{l',l} \Big) (\tau_{kN}\otimes I_M)\Big( \partial_i q \# E_{r,s}\otimes I_M\ P_{l',l} \Big)^* \right) \nonumber \\
	&= \frac{1}{k^2}\sum\limits_{1\leq l,l' \leq k} \tau_M\Bigg( e_l^* \otimes I_M\ (\tau_{N}\otimes I_k \otimes I_M)\Big( \partial_i q \# E_{r,s}\otimes I_M\ \Big)\  e_{l'}e_{l'}^*\otimes I_M\ \\
	& \quad \quad \quad \quad \quad \quad \quad \quad \quad (\tau_{N}\otimes I_k \otimes I_M)\Big( \partial_i q \# E_{r,s}\otimes I_M \Big)^* e_l\otimes I_M\ \Bigg) \nonumber \\	
	&= \frac{1}{k}\tau_k\otimes\tau_M\left( (\tau_{N}\otimes I_k \otimes I_M)\Big( \partial_i q \# E_{r,s}\otimes I_M\ \Big)\ (\tau_{N}\otimes I_k \otimes I_M)\Big( \partial_i q \# E_{r,s}\otimes I_M \Big)^* \right) \. \nonumber \\ \nonumber 
\end{align}

\noindent Hence by combining equations \eqref{comb1} and \eqref{comb2} we have proved that

\begin{align}
\label{interm2}
	&\frac{k}{M} \sum\limits_{\substack{ 1\leq l,l' \leq k \\ 1\leq u,v\leq M }} \var_R\left(F^q_{l,l',v,u}(R^{kN})\right) \nonumber \\
	&\leq \frac{1}{kN} \sum_i \sum_{1\leq r,s\leq kN} \E_R\left[ \tau_k\otimes\tau_M\left( (\tau_{N}\otimes I_{kM})\Big( \partial_i q \# E_{r,s}\otimes I_M\ \Big)\ (\tau_{N}\otimes I_{kM})\Big( \partial_i q \# E_{r,s}\otimes I_M \Big)^* \right) \right]
\end{align}

\noindent Moreover, let us remind that, with the convention $A\times (B\otimes C)\times D=(AB)\otimes (CD)$, we have (for $q=Q=Ae^{\i\beta y P}$ but with obvious changes for $q=T$ )

$$ \partial_s q = \partial_s A\  e^{\i \beta y P}\ B + \i \beta y A\ \int_0^1  e^{\i (1-u)\beta y P}\ \partial_s P\ e^{\i u\beta y P}\ B du + A\  e^{\i \beta y P}\ \partial_s B \.$$

\noindent Consequently \eqref{interm2} is a finite linear combination of terms  of the three following types $Q_N^i=\mathbb E_R[q_N^i]$, $1\le i\le 3$, with

	\begin{align*}
	q_N^1 = \frac{1}{kN} \sum_{1\leq r,s\leq kN}  \tau_k\otimes\tau_M\Big( &(\tau_{N}\otimes I_{kM})\Big( A_1 E_{r,s}\otimes I_M\ A_2 e^{\i \beta y P}\ A_3 \Big) \\
	&(\tau_{N}\otimes I_{kM})\Big( B_3 E_{s,r}\otimes I_M\ B_2 e^{-\i \beta y P}\ B_1 \Big) \Big)\,
	\end{align*}
	
	\item
	\begin{align*}
	q_N^2 = \frac{\beta y}{kN} \int_{0}^{1} \sum_{1\leq r,s\leq kN} \tau_k\otimes\tau_M\Bigg( &(\tau_{N}\otimes I_{kM})\Big( A_1 e^{\i (1-u)\beta y P}\ A_2 E_{r,s}\otimes I_M\ A_3 e^{\i u\beta y P}\ A_4 P_{l',l} \Big) \\
	&(\tau_{N}\otimes I_{kM})\Big( P_{l,l'} B_3\ E_{s,r}\otimes I_M\ B_2 e^{-\i \beta y P}\ B_1 \Big) \Bigg)  du \,
	\end{align*}
	
	\item
	\begin{align}
\label{lastmaj}
	q_N^3 = \frac{(\beta y)^2}{kN} \int_{0}^{1} \int_{0}^{1} \sum_{1\leq r,s\leq kN}  \tau_k\otimes&\tau_M\Bigg( (\tau_{N}\otimes I_{kM})\Big( A_1 e^{\i (1-u)\beta y P}\ A_2 E_{r,s}\otimes I_M\ A_3 e^{\i u\beta y P}\ A_4 \Big) \\
	&(\tau_{N}\otimes I_{kM})\Big( B_4 e^{-\i v\beta y P}\ B_3 E_{s,r}\otimes I_M\ B_2 e^{-\i (1-v)\beta y P}\ B_1 \Big) \Bigg)  du\ dv \nonumber ,
	\end{align}

\noindent where the $A_i$ and $B_i$ are monomial in $U_t^k$. Besides the coefficients of this linear combination only depend on $A,B$ and 
$P$.

\noindent We first show how to estimate $q_N^3$. Let us recall that we set $(e_i)_{1\leq i\leq N}$, $(f_i)_{1\leq i\leq k}$ and 
$(g_i)_{1\leq i\leq M}$ as the canonical basis of $\C^M$, $\C^k$ and $\C^N$. Then, for any matrices 
$A,B,C,D\in \M_N(\C)\otimes\M_k(\C)\otimes\M_M(\C)$, we have 

\begin{align}
\label{ccomb1}
&\sum_{1\leq r,s\leq kN} \tr_{kM}\Bigg( \tr_N\otimes I_{kM} \Big( A\ E_{r,s}\otimes I_M\ B \Big) \times \tr_N\otimes I_{kM}\Big( C\ E_{s,r}\otimes I_M\ D \Big) \Bigg) \\
& = \sum_{ 1\leq a,b,r_1,s_1 \leq N} \sum_{ 1\leq c,d,r_2,s_2 \leq k} \sum_{ 1\leq e,f,g,h \leq M} g_a^*\otimes f_c^*\otimes e_e^*\ A\  g_{r_1}\otimes f_{r_2}\otimes e_f \times g_{s_1}^*\otimes f_{s_2}^*\otimes e_f^*\ B\ g_a\otimes f_d\otimes e_g \nonumber\\
& \quad \quad \quad \quad \quad \quad \quad \quad \quad \quad \quad \quad \quad \quad \quad \quad \quad \times g_b^*\otimes f_d^*\otimes e_g^*\ C\  g_{s_1}\otimes f_{s_2}\otimes e_h \times g_{r_1}^*\otimes f_{r_2}^*\otimes e_h^*\ D\  g_b\otimes f_c\otimes e_e \nonumber\\
& \nonumber\\
& = \sum\limits_{\substack{ 1\leq a \leq N \\ 1\leq c,d \leq k \\ 1\leq e,f,g,h \leq M}} g_a^*\otimes f_c^*\otimes e_e^*\ A\  I_N\otimes I_k\otimes (e_f e_h^*)\ D\ I_N\otimes (f_c f_d^*)\otimes (e_e e_g^*)\ C\ I_N\otimes I_k\otimes (e_h e_f^*)\ B\ g_a\otimes f_d\otimes e_g \nonumber\\ 
&= \sum_{1\leq u,v\leq M} \tr_N\Big( I_N\otimes\tr_{kM}( A\ I_{kN}\otimes e_u e_v^*\ D )\  I_N\otimes\tr_{kM}( C\ I_{kN}\otimes e_v e_u^*\ B ) \Big) \. \nonumber
\end{align}

\noindent Let $K_M$ be a $GUE$ matrix of size $M$, independent of everything else. Performing a Gaussian integration by part, we get

\begin{align}
\label{ccomb2}
&\frac{1}{M} \sum_{1\leq u,v\leq M} \tr_N\Big( I_N\otimes\tr_{kM}( A\ I_{kN}\otimes e_u e_v^*\ D )\  I_N\otimes\tr_{kM}( C\ I_{kN}\otimes e_v e_u^*\ B ) \Big) \\
&= \E_K\Bigg[ \tr_N\Bigg( I_N\otimes\tr_{kM}\Big( A\ I_{kN}\otimes K_M\ D \Big)\ I_N\otimes\tr_{kM}\Big( C\ I_{kN} \otimes K_M\ B \Big) \Bigg) \Bigg] . \nonumber
\end{align}

\noindent Consequently by combining equations \eqref{ccomb1} and \eqref{ccomb2}, we have

\begin{align*}
q_N^3=\left( \frac{\beta y M}{N} \right)^2 \int_0^1 \int_0^1 \E_K\Bigg[ \tau_N\Bigg( &(I_N\otimes \tau_{kM}) \Big( A_1 e^{\i (1-u)\beta y P} A_2\ I_{kN}\otimes K_M\ B_2 e^{-\i (1-v)\beta y P} B_1 \Big) \\
&\times (I_N\otimes \tau_{kM})\Big( B_4 e^{-\i v\beta y P} B_3\ I_{kN}\otimes K_M\ A_3 e^{\i u\beta y P} A_4 \Big) \Bigg) \Bigg] du\ dv \.
\end{align*}

\noindent Since $P$ is self-adjoint, we know that for any real $r$, $\norm{e^{\i\ r P(U_t^k)}}=1$. Besides $\norm{I_{N}\otimes \tau_{kM}(A)}\leq \norm{A}$, thus we can bound $q_N^3$ in \eqref{lastmaj} by

\begin{equation}
\label{majvar}
	|q_N^3|\le \left(\frac{yM}{N}\right)^2 \norm{A_1}\norm{A_2}\norm{A_3}\norm{A_4} \norm{B_1}\norm{B_2}\norm{B_3}\norm{B_4}\ \E_K\left[ \norm{K_M}^2 \right] .
\end{equation}

\noindent Finally, by \cite{monomat}, $\E_K\left[ \norm{K_M}^2 \right] $ is bounded by $3$. One can bound similarly $q_N^1$ and $q_N^2$, 
the only difference on the final result is that we would have $1$ or $y$ instead of $y^2$. Finally after taking the expectation with respect to 
$R^{kN}$ in equation \eqref{majvar} and using Proposition \ref{bornenorme}, we deduce that there exists $S$ which only depends on $A,B$ 
and $P$, hence is independent of $N,M,y,t,\alpha$ or $\beta$, such that the covariance in \eqref{interm2} is bounded by 

$$\frac{k}{M} \sum\limits_{\substack{ 1\leq l,l' \leq k \\ 1\leq u,v\leq M }} \var_R\left(F^q_{l,l',v,u}(R^{kN})\right)\le  \frac{(1+y^2)M^2}{N^2} S\left(\norm{X^N},\norm{Z^{NM}}\right) .$$

\noindent 

Thus, we deduce that there exists a polynomial $H$ which only depends on $A,B,C,D$ and $P$ such that the first term in the right hand side  
of \eqref{CS2} is bounded by

\begin{equation}\label{fs1}\frac{k}{M} \sum\limits_{\substack{ 1\leq l,l' \leq k \\ 1\leq u,v\leq M }} \left|  \cov_R\left(F^Q_{l,l',u,v}(R^{kN}), F^T_{l',l,u,v}(R^{kN})\right)\right| \le  \frac{(1+y^2) M^2}{N^2} H\left(\norm{X^N},\norm{Z^{NM}}\right) . \end{equation}
This completes the proof of the Lemma in the general case.
For the specific case where $Z^{NM} = (I_N\otimes Y_1^M,\dots ,I_N\otimes Y_q^M)$ and that these matrices commute, we can get better 
estimate in equation \eqref{majvar} thanks to a refinement of equation \eqref{ccomb2}. Indeed if $A,B,C,D$ are monomials in $U_t^k$, then 
we can write $A = A_1\otimes A_2$ in $\M_{kN}(\C)\otimes\M_{M}(\C)$ and likewise for $B,C,D$ such that $A_2,B_2,C_2,D_2$ commute. 
Thus,

\begin{align*}
&\frac{1}{M} \sum_{1\leq u,v\leq M} \tr_N\Big( I_N\otimes\tr_{kM}( A\ I_{kN}\otimes e_u e_v^*\ D )\  I_N\otimes\tr_{kM}( C\ I_{kN}\otimes e_v e_u^*\ B ) \Big) \\
&= \frac{1}{M} \tr_N\Big( I_N\otimes\tr_{k}( A_1 D_1 )\  I_N\otimes\tr_{k}( C_1 B_1 ) \Big) \sum_{1\leq u,v\leq M} \tr_{M}(A_2\ e_u e_v^*\ D_2) \tr_{M}(C_2\ e_u e_v^*\ B_2) \\
&= \frac{1}{M} \tr_N\Big( I_N\otimes\tr_{k}( A_1 D_1 )\  I_N\otimes\tr_{k}( C_1 B_1 ) \Big) \tr_{M}(D_2 A_2 B_2 C_2) \\
&= \frac{1}{M} \tr_N\Big( I_N\otimes\tr_{k}( A_1 D_1 )\  I_N\otimes\tr_{k}( C_1 B_1 ) \Big) \tr_{M}(A_2 D_2 C_2 B_2) \\
&= \frac{1}{M} \tr_{NM}\Big( I_{NM}\otimes\tr_{k}( A D )\  I_{NM}\otimes\tr_{k}( C B ) \Big) \.
\end{align*}

\noindent By linearity and density this equality is true if we assume that $A,B,C,D$ are power series in $U_t^k$. Thus combining this equality 
with equation \eqref{ccomb1}, we get that in this case
\begin{equation*}
|q_N^3|\le \left(\frac{y}{N}\right)^2 \norm{A_1}\norm{A_2}\norm{A_3}\norm{A_4} \norm{B_1}\norm{B_2}\norm{B_3}\norm{B_4} \.
\end{equation*}

\noindent The same argument as in the general case applies and the proof follows. 

\end{proof}

In order to prove Lemma \ref{core}, we show in the following lemma that the term appearing in the second line of equation \eqref{intermedi} 
vanishes.

\begin{lemma}
	\label{nondiag}
	Let $U_t^k,P_{1,2},Q$ and $T$ be defined as in Lemma \ref{grcalc}, then $\P_{X^N}$-almost surely,
	
	$$  \lim_{k\to \infty} k^3 \tau_M\left( \E_R\left[(\tau_{kN}\otimes I_M)(QP_{1,2})\right] \E_R\left[(\tau_{kN}\otimes I_M)(TP_{1,2})\right] \right) = 0 \.$$
	
\end{lemma}

\begin{proof}

It is enough to show that given $y\in \R$ and monomial $A$ and $B$, we have

$$ \lim_{k\to \infty} k^{3/2} \E_R\left[(\tau_{kN}\otimes I_M)((A\ e^{\i y P}\ B)(U_t^k)\ P_{1,2})\right] = 0 \.$$

\noindent For this purpose, let us define for monomials $A,B$ and $y\ge 0$

$$ f_{A,B}(y) = \E_R\left[(\tr_{kN}\otimes I_M)((A\ e^{\i y P}\ B)(U_t^k)\ P_{1,2})\right] \.$$

\noindent  We want to show that $f_{A,B}$ goes to zero faster than $k^{-3/2}$ and first show that we can reduce the problem to the case 
$y=0$. To this end, we also define

$$ d_n(y) = \sup\limits_{\deg(A)+\deg(B)\leq n} \norm{f_{A,B}(y)} \.$$

\noindent We know thanks to Proposition \ref{bornenorme} that there exist constants $\alpha$ and $C$ such that for any $i$ and $n\leq \alpha kN/2$,

$$ \E\left[ \norm{R^{kN}_i}^n \right] \leq C^n \.$$

\noindent Consequently, $P_{X^N}$-almost surely, there exist constants $\gamma$ and $D$ (which do depend on, $N$, $\norm{X^N}$ and $\norm{Z^{NM}}$) such that for any $n\leq \gamma k$,
\begin{equation}
	\label{majorationstraightfor}
	d_n(y) \leq D^n \. 
\end{equation}

\noindent It is important to point out that this constant $D$ can be very large in $N$, it does not matter since in the end we will show 
that this quantity will go towards $0$ when $k$ goes to infinity and the other parameters such as $N,M$ or $y$ are fixed. Next, we define

$$g_{k,a}(y) = \sum_{0 \leq n\leq \gamma k} d_n(y) a^n \.$$

\noindent But if we set $c_L(P)$ to be  the coefficient associated to the monomial $L$ in $P$, $P=\sum c_L(P) L$,  we have

$$ \left| \frac{d f_{A,B}(y)}{dy} \right| \leq \sum_{L \text{ monomials}} |c_L(P)|\ d_{\deg(A)+\deg(B)+ \deg(L)}(y)  \.$$

\noindent Thus, for any $y\geq 0$, any monomials $A,B$ with $\deg (A)+\deg(B)=n$, 

$$ f_{A,B}(y) \leq f_{A,B}(0) + \sum_{L \text{ monomials}} |c_L(P)|\ \int_0^y d_{n + \deg(L)}(u)\  du \. $$

\noindent Therefore, we have for $y\geq 0$ and any $n\geq 0$,

$$ a^n d_n(y) \leq a^n d_n(0) + \sum_{L \text{ monomials}} |c_L(P)| a^{-\deg(L)}\ \int_0^y d_{n + \deg(L)}(u) a^{n+ \deg(L)} du \. $$

\noindent And with $\norm{.}_{a^{-1}}$ defined as in \eqref{normA}, thanks to \eqref{majorationstraightfor}, we have a constant $c_a$ independent of $k$ such that 

$$ g_{k,a}(y) \leq g_{k,a}(0) + c_a (aD)^{\gamma k} + \norm{P}_{a^{-1}} \int_0^y g_{k,a}(u) du \.$$

\noindent As a consequence of  Gronwall's inequality, we deduce that for $y\geq 0$,

\begin{equation}
	\label{gronwall}
	g_{k,a}(y) \leq \left( g_{k,a}(0) + c_a (aD)^{\gamma k} \right) e^{y\norm{P}_{a^{-1}}} \.
\end{equation}
Hence, it is enough   to find an estimate on $g_{k,a}(0)$.  First for any $j$, one can write $Z^{NM}_j = \sum_{1\leq u,v\leq N} E_{u,v}\otimes I_k\otimes A_{u,v}^j$ for some matrices $A_{u,v}^j$, then we define

$$ U_{N,k} = \Big( R^{kN}, X^N\otimes I_k , (E_{u,v}\otimes I_k)_{u,v} \Big) \, \quad c_n = \sup_{\deg(L)\leq n,\ L \text{ monomial}} \left| \E_R\left[\tr_{kN}(L(U_{N,k})\ P_{1,2})\right] \right| \.$$

\noindent Note that since we are taking the trace of $L(U_{N,k}) P_{1,2}$ with $P_{1,2} = I_N \otimes f_1 f_2^*\otimes I_M$, we have $c_0 = c_1 = 0$. We consider $K$ the supremum over $u,v,j$ of $\norm{A_{u,v}^j}$, we also naturally assume that $K\geq 1$. Thus, since 

$$ Z^{NM}_j = \sum_{1\leq u,v\leq N} E_{u,v}\otimes I_k\otimes A_{u,v}^j \,\quad X_t^N = e^{-t/2}X^N\otimes I_k + (1-e^{-t})^{1/2} R^{kN} \, $$

\noindent if $L$ is a monomial in $U_t^k = (X_t^N\otimes I_M,Z^{NM}\otimes I_k,{Z^{NM}}^*\otimes I_k)$ of degree $n$, then we can 
view $L(U_t^k)$ as a sum of at most $2^n N^{2n}$ monomials in $e^{-t/2}X^N\otimes I_k$, $(1-e^{-t})^{1/2} R^{kN}$, 
$E_{u,v}\otimes I_k\otimes A_{u,v}^j$, $E_{v,u}\otimes I_k\otimes {A_{u,v}^j}^*$. Consequently, since $\sup_{u,v,j} \norm{A^j_{u,v}} \leq K$, 
we have
\begin{align*}
	\norm{ \E_R\left[\tr_{kN}\otimes I_M(L(U_t^k) P_{1,2})\right] } \leq 2^n N^{2n} K^n c_n \.
\end{align*}

\noindent Thus, if we set 
$$f_p(a) = \sum_{0 \leq n\leq p} c_n a^n \, $$ 

\noindent we have 
\begin{equation}
	\label{gppf}
	g_{k,a}(0) \leq f_{\gamma k}(2N^2K a) \.
\end{equation}

Now we need to study the behaviour of $f_k(a)$ when $k$ goes to infinity for $a$ small enough. In order to do so, let us consider a 
monomial $L$ in $U_{N,k}$. Since $X^N\otimes I_k$ and $E_{u,v}\otimes I_k$ commute with $P_{1,2}$, one can assume that 
$L=R_i^{kN} S$ for some $i$ (unless $L$ is a monomial in $X^N\otimes I_k$ and $E_{u,v}\otimes I_k$ in which case 
$\tr_{kN}(L P_{1,2}) = 0$), thus thanks to Schwinger-Dyson equation (see Proposition \ref{SD}),

\begin{equation}\label{asd}
\E_R\left[\tr_{kN}(L P_{1,2})\right] = \frac{1}{N k} \E_R\left[ \tr_{kN}\otimes\tr_{kN}( \partial_i (S P_{1,2})) \right] = \frac{1}{Nk} \sum_{S = UR_i V} \E[\tr_{Nk}(U) \tr_{Nk}(VP_{1,2}) ] . \\
\end{equation}

\noindent To use this Schwinger-Dyson equation as an inductive bound we shall use Poincar\'e inequality to bound the covariance in the above right hand side.We hence compute for any monomial $V$,

\begin{align}
\norm{\nabla \tr_{kN}(VP_{1,2})}_2^2 &= \frac{1}{Nk} \sum_i \sum_{r,s}\tr_{kN}(\partial_s V\# E_{r,s}P_{1,2}) \tr_{kN}(\partial_s V\# E_{s,r}P_{1,2})^* \nonumber \\
&= \sum_i\sum_{V=AR_iB,V=CR_iD} \frac{1}{Nk} \tr_{kN}(BP_{1,2}AC^*P_{1,2}^*D^*)\label{lkj}
\end{align}

\noindent Thus with $\Theta = \max \left\{ C, \norm{X^N}, 1 \right\}$, since $P_{1,2}$ is of rank $N$, we get
$$ \text{Var}_R(\tr_{kN}(V P_{1,2})) \leq \frac{1}{k}(\deg V)^2 \Theta^{2\deg V} . $$

\noindent Likewise, for any monomial $U$, 
$$ \text{Var}_R(\tr_{kN}(U)) \leq (\deg U)^2 \Theta^{2\deg U} . $$

\noindent Therefore,  if $n$ is the degree of $L$, we deduce from \eqref{lkj}, \eqref{asd} and Poincar\'e inequality that 

\begin{align*}
\left| \E_R\left[\tr_{kN}(L P_{1,2})\right] \right| &\leq \frac{1}{ k^{3/2}N} \sum_{i=0}^{n-2} i(n-2-i) \Theta^n + \sum_{S=UR_iV} \left| \frac{1}{Nk} \E_R[\tr_{kN}(U)] \E_R[\tr_{kN}(V P_{1,2})] \right| \\
&\leq \frac{n^3 \Theta^n}{k^{3/2}N} + \sum_{S=UR_iV} \left| \E_R[\tr_{kN}(V P_{1,2})] \right| \Theta^{\deg U} \. \\
\end{align*}

\noindent By replacing $D$ by $\max \{D,\Theta\}$, we can always assume that $\Theta < D$. We also bound $N^{-1}$ by $1$, thus for $n\geq 2$,

$$ c_n \leq \frac{n^3 D^n}{k^{3/2}} + \sum_{i=0}^{n-2} c_i D^{n-2-i} \.$$

\noindent We use this estimate to bound $f_g(a)$ with 
$g$ such that $g^3 D^g \leq \sqrt{k}$. Since $c_0=c_1=0$ and for any $n\leq g$, $n^3 D^n k^{-3/2} \leq k^{-1}$, we have

\begin{equation*}
f_g(a) = \sum_{n=2}^g c_n a^n \leq \frac{1}{k}\times \frac{a^2-a^{g+1}}{1-a} + a^2 \sum_{m=0}^{g-2} \sum_{n=0}^m c_i D^{n-i} a^n 
\leq \frac{1}{k}\times \frac{a^2}{1-a} + a^2 \frac{f_g(a)}{1-Da} \.
\end{equation*}

\noindent Thus, for $a$ small enough,

$$ f_g(a) \leq \frac{(1-Da)a^2}{(1-a)(1-Da-a^2)}\times \frac{1}{k} \. $$

\noindent Besides, we want $g$ such that $g^3 D^g \leq \sqrt{k}$, hence we can take $g$ the integer part of $\frac{\ln k}{2(\ln D +3)}$. 
Since by definition we have $c_n\leq \Theta^n$, this also means that $c_n\leq D^n$, thus

\begin{equation*}
\sum_{g<n\leq \gamma k} c_n a^n \leq \sum_{n>g} (Da)^n \leq \frac{(Da)^{g+1}}{1-Da} \leq k^{\frac{\ln(Da)}{2(\ln D +3)}}\times \frac{1}{1-Da} \.
\end{equation*}

\noindent Thus, if we fix $a$ small enough, $f_{\gamma k} (a)= O(1/k)$. Hence, we deduce from  \eqref{gppf} that for $a$ small enough (depending on $N,K$ but not $k$)

$$ g_{k,a}(0)\leq f_k(2N^2K a).$$

\noindent Therefore, by plugging this inequality in \eqref{gronwall}, we obtain for $a$ small enough and $y\geq 0$, $g_{k,a}(y)= O(1/k)$. 
By replacing $P$ by $-P$, we have for $a$ small enough and any $y\in\R$, $g_{k,a}(y)= O(1/k)$. By definition of $d_n$, we have

$$ k^{3/2}\left| \E_R\left[(\tau_{kN}\otimes I_M)((A\ e^{\i y P}\ B)(U_t^k)\ P_{1,2})\right]\right| \leq k^{1/2} d_{\deg(A)+\deg(B)}(y) \.$$

\noindent Hence for $k$ large enough,

$$ k^{3/2} \left| \E_R\left[(\tau_{kN}\otimes I_M)((A\ e^{\i y P}\ B)(U_t^k)\ P_{1,2})\right]\right| \leq k^{1/2} g_{k,a}(y) a^{-\deg(A)-\deg(B)} \, $$
which goes to zero as $k$ goes to infinity since $g_{k,a}(y)=O(\frac{1}{k})$.

\end{proof}

\vspace{1cm}

We can now prove Theorem \ref{imp1}.

\begin{proof}[Proof of Theorem \ref{imp1}] It is based  on Theorem \ref{imp2}. To use it,  we want to take the Fourier transform of $f$ and 
use Fourier inversion formula. However we did not assume that $f$ is integrable. Thus the first step of the proof is to show that up to a 
term of order $e^{-N}$, we can assume that $f$ has compact support. Thanks to Proposition \ref{bornenorme}, there exist constants $D$ 
and $\alpha$ such that for any $N$ and $i$, for any $u\geq 0$,

$$ \P\left(\norm{X^N_i}\geq u+D \right) \leq e^{-\alpha u N} \. $$

\noindent Thus, there exist constants $C$ and $K$, independent of $M,N,P$ and $f$, such that

\begin{align*}
&\left| \E\left[\frac{1}{MN}\tr\Big(f\left(P\left( X^N\otimes I_M,Z^{NM},{Z^{NM}}^*\right)\right)\Big) \1_{\left\{ \exists i, \norm{X_i^N} > D+1 \right\}} \right] \right| \\
&\leq \E\left[\norm{f\left(P\left( X^N\otimes I_M,Z^{NM},{Z^{NM}}^*\right)\right)} \1_{\left\{ \exists i, \norm{X_i^N} > D+1 \right\}} \right] \\
&\leq \norm{f}_{\infty} \P\left( \exists i, \norm{X_i^N} > D+1 \right) \\
&\leq C \norm{f}_{\infty} e^{-K N} \.\\
\end{align*}

\noindent There exists a polynomial $H$ which only depends on $P$ such that

$$ \norm{P\left(X^N\otimes I_M,Z^{NM},{Z^{NM}}^*\right)} \1_{\left\{ \forall i, \norm{X_i^N}\leq D+1 \right\}} \leq H\left( \norm{Z^{NM}} \right) \. $$

\noindent We can also assume that $\norm{P(x\otimes I_M,Z^{NM},{Z^{NM}}^*)} \leq H\left( \norm{Z^{NM}} \right)$. 
We take $g$ a $\mathcal{C}^{\infty}$-function which takes value $1$ on $[-H\left( \norm{Z^{NM}} \right),H\left( \norm{Z^{NM}} \right)]$, 
$0$ on $[-H\left( \norm{Z^{NM}} \right)-1,H\left( \norm{Z^{NM}} \right)+1]^{c}$ and belongs to  $[0,1]$ elsewhere. From the bound above, 
we deduce

\begin{align}
\label{redux}
&\left| \E\left[\frac{1}{MN}\tr\Big(f\left(P\left( X^N\otimes I_M,Z^{NM},{Z^{NM}}^*\right)\right)\Big)\right] - \tau\Big(f\left(P\left( x\otimes I_M,Z^{NM},{Z^{NM}}^*\right)\right)\Big) \right| \nonumber \\
&\leq \Bigg| \E\left[\frac{1}{MN}\tr\Big(f\left(P\left( X^N\otimes I_M,Z^{NM},{Z^{NM}}^*\right)\right)\Big) \1_{\left\{ \forall i, \norm{X_i^N}\leq D+1 \right\}}\right] \nonumber \\
&\quad \quad - \tau\Big(f\left(P\left( x\otimes I_M,Z^{NM},{Z^{NM}}^*\right)\right)\Big) \Bigg| +C  \norm{f}_{\infty} e^{-K N} \\
&\leq \Bigg| \E\left[\frac{1}{MN}\tr\Big((fg)\left(P\left( X^N\otimes I_M,Z^{NM},{Z^{NM}}^*\right)\right)\Big)\right] \nonumber \\
& \quad \quad - \tau\Big((fg)\left(P\left( x\otimes I_M,Z^{NM},{Z^{NM}}^*\right)\right)\Big) \Bigg| + 2 C   \norm{f}_{\infty} e^{-K N} . \nonumber 
\end{align}

\noindent Since $fg$ has compact support and can be differentiated six times, we can take its Fourier transform and then invert it so that with the convention $ \hat{h}(y) = \frac{1}{2\pi} \int_{\R} h(x) e^{-\i xy} dx$, we have

$$ \forall x\in\R,\quad (fg)(x) = \int_{\R} e^{\i xy} \widehat{fg}(y)\ dy \. $$

\noindent Besides, since if $h$ has compact support bounded by $K$ then $\norm{\hat{h}}_{\infty} \leq 2K \norm{h}_{\infty} $, we have

\begin{align*}
	\int_{\R} (|y|+y^4)\ \left| \widehat{fg}(y) \right|\ dy &\leq \int_{\R} \frac{|y|+|y|^3+y^4+y^6}{1+y^2}\ \left| \widehat{fg}(y) \right|\ dy \\
	&\leq \bigintss_{\R} \frac{\left| \widehat{(fg)^{(1)}}(y) \right| + \left| \widehat{(fg)^{(3)}}(y) \right| + \left| \widehat{(fg)^{(4)}}(y) \right| + \left| \widehat{(fg)^{(6)}}(y) \right|}{1+y^2}\ dy \\
	&\leq 2 \left( H\left(\norm{Z^{NM}}\right) + 1\right) \norm{fg}_{\mathcal{C}^6} \int_{\R} \frac{1}{1+y^2}\ dy \\
	&\leq C \left( H\left(\norm{Z^{NM}}\right) + 1\right) \norm{f}_{\mathcal{C}^6} \,
\end{align*}

\noindent for some absolute constant $C$. Hence $fg$ satisfies the hypothesis of Theorem \ref{imp2} with $\mu(dy) = \widehat{fg}(y) dy$. 
Therefore,  combining with equation \eqref{redux}, we conclude that

\begin{align*}
&\left| \E\left[\frac{1}{MN}\tr\Big(f\left(P\left( X^N\otimes I_M,Z^{NM},{Z^{NM}}^*\right)\right)\Big)\right] - \tau\Big(f\left(P\left( x\otimes I_M,Z^{NM},{Z^{NM}}^*\right)\right)\Big) \right| \nonumber \\
&\leq \norm{f}_{\infty} e^{-K N} + \frac{M^2}{N^2} L_P\left(\norm{Z^{NM}}\right) \int_{\R} (|y|+y^4)\ \left| \widehat{fg}(y) \right|\ dy \\
&\leq \frac{M^2}{N^2} \left( C L_P\left(\norm{Z^{NM}}\right) \left( H\left(\norm{Z^{NM}}\right) + 1\right) + e^{-K N} \right) \norm{f}_{\mathcal{C}^6} \. \\
\end{align*}

\end{proof}

\section{Consequences}
\label{proofcoro}
In this section, we deduce Corollaries \ref{stieljes} and \ref{LU}, as  well as  Theorems \ref{strongconv} and \ref{concentration}.

\subsection{Proof of Corollary \ref{stieljes}}

We could directly apply Theorem \ref{imp1} to $f_z : x\to (z-x)^{-1}$, however we have $\norm{f}_{\mathbb{C}^6} = O\left((\Im z)^7 \right) $ 
when we want an exponent $5$. Since $\overline{G_{P(x)}(z)} = G_{P(x)}(\overline{z})$ we can assume that $\Im z < 0$, but then

$$ f_z(x) = \int_{0}^{\infty} e^{\i x y}\ (\i e^{-\i y z})\ dy \.$$

\noindent Consequently, with $\mu_z(dy) = \i e^{-\i y z}\ dy $, we have

\begin{align*}
	\int_{0}^{\infty} (y+y^4)\ d|\mu_z|(y) &= \frac{1}{|\Im z|^2} + \frac{24}{|\Im_z|^5} \.
\end{align*}

\noindent Thus, by applying Theorem \ref{imp2} with $Z^{NM} = \left(I_N\otimes Y_1^M,\dots,I_N\otimes Y_p^M\right)$, $P$ and $f_z$, 
we have

$$ \left| \E\left[ G_{P(X^N\otimes I_M, I_N\otimes Y^M)}(z) \right] - G_{P(x\otimes I_M, I_N\otimes Y^M)}(z) \right| \leq \frac{M^2}{N^2} L_P\left(\norm{Z^{NM}}\right) \int_{\R} (1+y^4)\  d|\mu_z|(y) \.$$

\noindent Now since $\norm{Z^{NM}} = \left(\norm{Y^M_1},\dots,\norm{Y_p^M}\right)$ which does not depend on $N$, we get the 
desired estimate

$$ \left| \E\left[ G_{P(X^N)}(z) \right] - G_{P(x)}(z) \right| \leq \frac{M^2}{N^2} L_P\left(\norm{Y^M_1},\dots,\norm{Y_p^M}\right)  \left(\frac{1}{|\Im z|^2} + \frac{24}{|\Im_z|^5}\right) \. $$

\subsection{Proof of Corollary \ref{LU}}

Let $f:\R\to\R$ be a Lipschitz function uniformly bounded by $1$ and with Lipschitz constant at most $1$. We want to bound from above 
the quantity

\begin{equation}
\label{initi}
	\Delta_{N,M}(f)=\Bigg| \E\left[\frac{1}{MN}\tr_{NM}\Big(f\left(P\left(X^N\otimes I_M,I_N\otimes Y_M\right)\right)\Big)\right] - \tau\otimes\tau_M\Big(f\left(P\left(x\otimes I_M,I_N\otimes Y_M\right)\right)\Big) \Bigg| 
\end{equation}

\noindent Firstly, one can see that with the same argument as in the proof of Theorem \ref{imp1} (in particular equation \eqref{redux}), we 
can assume that the support of $f$ is bounded by a constant $S = H(\norm{Y^M})$ for some polynomial $H$ independent of everything. 
However, we cannot apply directly Theorem \ref{imp1} since $f$ is not regular enough. In order to deal with this issue we use the 
convolution with Gaussian random variable, thus let $G$ be a centered Gaussian random variable, we set

$$ f_{\varepsilon} : x\to \E[f(x+\varepsilon G)] \.$$

\noindent Since $f$ has Lipschitz constant $1$, we have for any $x\in\R$,

$$ \left| \E[f(x+\varepsilon G)] -f(x)\right| \leq \varepsilon \.$$

\noindent Since $f_{\varepsilon}$ is regular enough we could now apply Theorem \ref{imp1}, however we a get better result by using 
Theorem \ref{imp2}. Indeed we have

\begin{align*}
	f_{\varepsilon}(x) &= \frac{1}{\sqrt{2\pi}} \int_{\R} f(x+\varepsilon y) e^{-y^2/2}\  dy \\
	&= \frac{1}{\sqrt{2\pi}} \int_{\R} f(y) \frac{e^{-\frac{(x-y)^2}{2\varepsilon^2}}}{\varepsilon}\  dy \\
	&= \frac{1}{2\pi} \int_{\R} f(y) \int_{\R} e^{\i (x-y) u} e^{-(u\varepsilon)^2/2}\ du\ dy \.
\end{align*}

\noindent Since the support of $f$ is bounded, we can apply Fubini's Theorem:

\begin{align*}
f_{\varepsilon}(x) &= \frac{1}{2\pi} \int_{\R} e^{\i ux} \int_{\R} f(y) e^{- \i y u}  dy\  e^{-(u\varepsilon)^2/2}\  du \.
\end{align*}

\noindent And so with the convention $ \hat{h}(u) = \frac{1}{2\pi} \int_{\R} h(y) e^{-\i uy} dy$, we have

$$ f_{\varepsilon}(x) = \int_{\R} e^{\i ux} \hat{f}(u)  e^{-(u\varepsilon)^2/2} du \.$$

\noindent Thus, if we set $\mu_{\varepsilon}(dy) = \hat{f}(y)  e^{-(y\varepsilon)^2/2} dy$, then, since $\norm{f}_{\infty} \leq 1$,

$$ \int_{\R} (1+y^4) d|\mu_{\varepsilon}|(y) \leq 2S \int_{\R} (1+y^4)e^{-y^2/2}\ dy\ \varepsilon^{-5} \. $$

\noindent Consequently, we can apply Theorem \ref{imp2} with $f_{\varepsilon}$ and since $\norm{f-f_{\varepsilon}}_{\infty}\leq \varepsilon$, 
there exists a polynomial $R_P$ such that the difference in  \eqref{initi} can be bounded by:

$$\Delta_{N,M}(f)\le  2\varepsilon + R_P\left(\norm{Y^M}\right) \frac{M^2}{N^2 \varepsilon^5}\,. $$
We finally choose  $\varepsilon = N^{-1/3}$ to get the desired bound
$$\Delta_{N,M}(f)\le 2 R_P\left(\norm{Y^M}\right) \frac{M^2}{N^{1/3}} \. $$

\subsection{Proof of Theorem \ref{strongconv}}

Firstly, we need to define properly the operator norm of tensor of $\CC^*$-algebras. When writing the proof it appears that we should 
work with the minimal tensor product.

\begin{defi}
	\label{mini}
	Let $\A$ and $\B$ be $\CC^*$-algebras with faithful representations $(H_{\A},\phi_{\A})$ and $(H_{\B},\phi_{\B})$, then if $\otimes_2$ 
	is the tensor product of Hilbert spaces, $\A\otimes_{\min}\B$ is the completion of the image of $\phi_{\A}\otimes\phi_{\B}$ in 
	$B(H_{\A}\otimes_2 H_{\B})$ for the operator norm in this space. This definition is independent of the representations that we fixed.
\end{defi}

The following two lemmas are  well  known facts in operator algebra. The first one is Lemma 4.1.8 from \cite{ozabr}:
\begin{lemma}
	\label{faith}
	Let $(\A,\tau_{\A})$ and $(\B,\tau_{\B})$ be $\CC^*$-algebra with faithful traces, then $\tau_{\A}\otimes\tau_{\B}$ extends uniquely 
	to a faithful trace $\tau_{\A}\otimes_{\min}\tau_{\B}$ on $\A\otimes_{\min}\B$. 
\end{lemma}

We did not find  a reference with an explicit proof for the  following Lemma, so we give our own. In order to learn more about this 
second lemma, especially how to weaken the hypothesis, we refer to \cite{pisier}.

\begin{lemma}
	\label{tensorconv}
	Let $\mathcal{C}$ be an exact $\mathcal{C}^*$-algebra endowed with a faithful state $\tau_{\CC}$, let $Y^N \in \mathcal{A}_N$ be a 
	sequence of family of noncommutative random variables in a $\mathcal{C}^*$-algebra $\mathcal{A}_N$ which converges strongly 
	towards a family $Y$ in a $\mathcal{C}^*$-algebra $\mathcal{A}$ endowed with a faithful state $\tau_{\A}$. Let $S\in \mathcal{C}$ be a 
	family of noncommutative random variables, then the family $(S\otimes 1, 1\otimes Y^N)$ converges strongly in distribution towards the 
	family $(S\otimes 1, 1\otimes Y)$.
\end{lemma}

\begin{proof}
	The following sets
	
	$$ \mathcal{M} = \left\{ (x_N)_{N\in \N}\ \middle| \ x_N\in\A_N, \sup_{N\geq 0} \norm{ x_N } <\infty \right\} \, $$ 
	
	$$ \mathcal{I} = \left\{ (x_N)_{N\in \N}\in \mathcal{M}\ \middle| \ \lim_{N\to \infty} \norm{ x_N } = 0 \right\} \, $$ 
	
	\noindent are $\mathcal{C}^*$-algebras for the norm $\norm{x} = \sup_{N\geq 0} \norm{ x_N }$. We also define 
	
	$$ \mathcal{B} = \mathcal{C}^*\left( (Y_N)_{N\in \N}\ , \mathcal{I} \right) \,$$
	
	\noindent the $\mathcal{C}^*$-algebra generated by $\mathcal{I}$ and the family $(Y_N)_{N\in \N}$. Since $\mathcal{I}$ is a 
	closed ideal of $\mathcal{B}$, by Theorem 3.1.4	of \cite{murphy}, $\mathcal{B}/\mathcal{I}$ is a $\mathcal{C}^*$-algebra for 
	the quotient norm. We naturally have the following exact sequence
	
	$$ 0 \to \mathcal{I} \to \mathcal{B} \to \mathcal{B}/\mathcal{I} \to 0 \.$$
	
	\noindent And by hypothesis, since $\mathcal{C}$ is exact, we have the following exact sequence
	
	$$ 0 \to \mathcal{C}\otimes_{\min} \mathcal{I} \to \mathcal{C}\otimes_{\min} \mathcal{B} \to \mathcal{C}\otimes_{\min} (\mathcal{B}/\mathcal{I}) \to 0 \. $$
	
	\noindent By definition, this means that $ (\mathcal{C}\otimes_{\min} \mathcal{B})/ (\mathcal{C}\otimes_{\min} \mathcal{I}) \simeq \mathcal{C}\otimes_{\min} (\mathcal{B}/\mathcal{I}) $. If $\pi_{\mathcal{I}}$ is the quotient map from $\mathcal{B}$ to $\mathcal{B}/\mathcal{I}$, the isomorphism between these two spaces is
	
	$$ f: x + \mathcal{C}\otimes_{\min} \mathcal{I} \mapsto \text{id}_{\mathcal{C}}\otimes_{\min} \pi_{\mathcal{I}}(x) \.$$	
	
	\noindent Hence 
	\begin{equation}
	\label{isomor}
	f(P\big( 1\otimes (Y_N)_{N\in \N}, S\otimes 1 \big) + \mathcal{C}\otimes_{\min} \mathcal{I}) = P\big( 1\otimes ((Y_N)_{N\in \N} + \mathcal{I}), S\otimes 1 \big) \.
	\end{equation}
	
	\noindent Let $(H,\varphi)$ be a faithful representation of $\mathcal{C}$, and $(H_N,\varphi_N)$ a faithful representation of 
	$\A_N$. The direct sum $(\bigoplus_{N\in\N} H_N, \bigoplus_{N\in\N} \varphi_N)$ is a faithful representation of $\mathcal{M}$ and 
	consequently of $\mathcal{B}$ too. More precisely, it is defined by
	
	$$ \bigoplus_{N\in\N} H_N = \left\{ (x_N)_{N\in\N}\ \middle|\ x_N\in H_N, \sum_N \norm{x_N}_2^2 <\infty \right\} \. $$
	
	\noindent Consequently, by definition of the spatial tensor product, it is the completion of the algebraic tensor 
	$\mathcal{C} \otimes \mathcal{B}$ in the operator space $B\left(H \otimes_2 (\oplus_{N} H_N)\right)$ endowed with the operator 
	norm. The notation $\otimes_2$ means that we completed the algeraic tensor $H\ \overline{\otimes}\ (\oplus_{N} H_N)$ to make 
	it a Hilbert space. It is important to see that this space is isomorphic to $\oplus_{N} (H \otimes_2 H_N)$, indeed it means that if $P$ 
	is a non-commutative polynomial, then
	
	$$ \norm{P\big( 1\otimes (Y_N)_{N\in \N}, S\otimes 1 \big)}_{\mathcal{C}\otimes_{\min} \mathcal{B}} =\ \sup_{N\geq 0}\ \norm{P\big( 1\otimes Y_N, S\otimes 1 \big)}_{\mathcal{C}\otimes_{\min} \mathcal{\A_N}} \.$$
	
	\noindent Consequently by using the definition of the quotient norm, we have
	\begin{equation}
	\label{normmm}
	\norm{P\big( 1\otimes (Y_N)_{N\in \N}, S\otimes 1 \big) + \mathcal{C}\otimes_{\min} \mathcal{I} }_{(\mathcal{C}\otimes_{\min} \mathcal{B})/ (\mathcal{C}\otimes_{\min} \mathcal{I})} =\ \limsup_{N\to \infty}\ \norm{P\big( 1\otimes Y_N, S\otimes 1 \big)}_{\mathcal{C}\otimes_{\min} \mathcal{\A_N}} \.
	\end{equation}
	
	\noindent Since $f$ is a $\mathcal{C}^*$-algebra isomorphism, thanks to \eqref{isomor}, we have
	$$ \norm{P\big( 1\otimes (Y_N)_{N\in \N}, S\otimes 1 \big) + \mathcal{C}\otimes_{\min} \mathcal{I} }_{(\mathcal{C}\otimes_{\min} \mathcal{B})/ (\mathcal{C}\otimes_{\min} \mathcal{I})} = \norm{P\big( 1\otimes ((Y_N)_{N\in \N} + \mathcal{I}), S\otimes 1 \big)}_{\mathcal{C}\otimes_{\min} (\mathcal{B}/\mathcal{I})} \. $$
	
	\noindent By definition of $\mathcal{I}$, if $P$ is a non-commutative polynomial, we have
	
	$$ \norm{P((Y_N)_{N\in \N} + \mathcal{I})}_{\mathcal{B}/\mathcal{I}} = \norm{P(Y)}_{\A} \.$$
	
	\noindent For our purposes, we can assume that $\A = \mathcal{C}^*(Y)$. Therefore the map 
	
	$$ P((Y_N)_{N\in \N}+ \mathcal{I}) \in \C\langle (Y_N)_{N\in \N}+ \mathcal{I} \rangle \mapsto P(Y) \in \C\langle Y \rangle $$
	
	\noindent is well-defined and is an isometry. Thus since $\C\langle (Y_N)_{N\in \N}+ \mathcal{I} \rangle$ is dense in $\mathcal{B}/\mathcal{I}$ and $\C\langle Y \rangle$ is dense in $\A$, this isometry extends into an isomorphism between $\mathcal{B}/\mathcal{I}$ and $\A$. Consequently 
	
	$$ \norm{P\big( 1\otimes ((Y_N)_{N\in \N} + \mathcal{I}), S\otimes 1 \big)}_{\mathcal{C}\otimes_{\min} (\mathcal{B}/\mathcal{I})} = \norm{P\big( 1\otimes Y, S\otimes 1 \big)}_{\mathcal{C}\otimes_{\min} \A} \.$$
	
	\noindent Thus, combined with \eqref{normmm}, we have
	
	\begin{equation}\label{limsup} \limsup_{N\to \infty}\ \norm{P\big( 1\otimes Y_N, S\otimes 1 \big)}_{\mathcal{C}\otimes_{\min} \mathcal{\A_N}} = \norm{P\big( 1\otimes Y, S\otimes 1 \big)}_{\mathcal{C}\otimes_{\min} \A} \. \end{equation}
	
	\noindent Finally let $f$ be a function which takes value $0$ on $(-\infty,\norm{P( 1\otimes Y, S\otimes 1 )}_{\mathcal{C}\otimes_{\min} \A} - \varepsilon]$ and positive value on $(\norm{P( 1\otimes Y, S\otimes 1 )}_{\mathcal{C}\otimes_{\min} \A} - \varepsilon,\infty)$. Since the family $(S\otimes 1, 1\otimes Y^N)$ converges clearly in distribution towards the family $(S\otimes 1, 1\otimes Y)$, we have
	
	$$ \lim_{N\to \infty} \tau_{\CC}\otimes_{\min}\tau_{\A_N}\Big(f(P(1\otimes Y_N, S\otimes 1))\Big) = \tau_{\CC}\otimes_{\min}\tau_{\A}\Big(f(P(1\otimes Y, S\otimes 1))\Big) \.$$
	
	\noindent Thanks to Lemma \ref{faith}, we know that $\tau_{\CC}\otimes_{\min}\tau_{\A}$ is faithful, consequently 
	
	$$ \tau_{\CC}\otimes_{\min}\tau_{\A}\Big(f(P(1\otimes Y, S\otimes 1))\Big) >0 \. $$
	
	\noindent This means that for $N$ large enough, $\tau_{\CC}\otimes_{\min}\tau_{\A_N}\Big(f(P(1\otimes Y_N, S\otimes 1))\Big)>0$, 
	thus for any $\varepsilon>0$, 
	
	$$ \liminf_{N\to \infty}\ \norm{P\big( 1\otimes Y_N, S\otimes 1 \big)}_{\mathcal{C}\otimes_{\min} \mathcal{\A_N}} \geq \norm{P\big( 1\otimes Y, S\otimes 1 \big)}_{\mathcal{C}\otimes_{\min} \A} - \varepsilon \. $$
This allows to conclude with \eqref{limsup} that
	
	$$ \lim_{N\to \infty}\ \norm{P\big( 1\otimes Y_N, S\otimes 1 \big)}_{\mathcal{C}\otimes_{\min} \mathcal{\A_N}} = \norm{P\big( 1\otimes Y, S\otimes 1 \big)}_{\mathcal{C}\otimes_{\min} \A} \. $$
	
\end{proof}

In order to prove Theorem \ref{strongconv} we use well-known concentration properties of Gaussian random variable coupled with 
an estimation of the expectation, let us begin by stating the concentration properties (see \cite{alice} Lemma 2.3.3).

\begin{prop}
	\label{concentr}
	Let $G$ be a Lipschitz function on $\R^n$ with Lipschitz constant $K$ for the $\ell^{2}$- norm 
	$ \norm{\gamma}_2 = ( \sum_i \gamma_i^2)^{1/2} $, $\gamma=(\gamma_1\etc \gamma_n)$ independent centered Gaussian random 
	variable of variance $1$. Then for all $\delta>0$,
	
	$$ \P\left( G(\gamma) - \E[G(\gamma)] \geq \delta \right) \leq e^{-\frac{\delta^2}{2 K^2}} \. $$
\end{prop}

\vspace{1cm}

In our situation, we have $p$ independent GUE matrices $(X^{N,i})_s$ of size $N$, hence we fix $\gamma$ the random vector of 
size $dN^2$ which consists of the union of $(\sqrt{N}X^{N,i}_{s,s})_{i,s}$, $(\sqrt{2N}\ \Re{X^{N,i}_{s,r}})_{s<r,i}$ and  
$(\sqrt{2N}\ \Im{X^{N,i}_{s,r}})_{s<r,i}$ which are indeed centered Gaussian random variable of variance $1$ as stated in 
Definition \ref{GUEdef}. We would like to apply Proposition \ref{concentr} to

$$ G_N(\gamma) = \norm{P^*P(X^N\otimes I_M, Z^{NM}, {Z^{NM}}^* )} \.$$
However $G_N$ is not  Lipschitz on $\R^{dN^2}$ because of its polynomial behaviour at infinity. Hence we cannot use directly 
Proposition \ref{concentr}. The following lemma is a well-known tool for this kind of situation, the proof can be found in 
\cite{lemmaconcentration}, Lemma 5.9.

\begin{lemma}
	\label{saintflour}
	Let $(X,d)$ be a metric space and $\mu$ a probability measure on $(X,d)$ which satisfies a concentration inequality, i.e. for 
	all $f:X\to\R$ with Lipschitz constant $|f|_{\mathcal{L}}$, for all $\delta>0$, 
	$$ \mu\Big(|f - \mu(f)| \geq \delta\Big) \leq e^{-g\left(\frac{\delta}{|f|_{\mathcal{L}}}\right)} $$
	
	\noindent for some increasing function $g$ on $\R^+$. Let $B$ be a subset of $X$ and $|f|_{\mathcal{L}}^B$ be the Lipschitz constant 
	of $f$ as a function from $B$ to $\R$. Let $\delta(f) = \mu(\ \1_{x\in B^{c}} ( |f(x)| + \sup_{u\in B} |f(u)| + |f|_{\mathcal{L}}^B d(x,B) )\ )$, 
	then
	$$ \mu\Big( |f - \mu(f)| \geq \delta + \delta(f) \Big) \leq \mu(B^c) + e^{-g\left(\frac{\delta}{|f|_{\mathcal{L}}^B}\right)} \.$$
	
\end{lemma}

We can now prove the concentration inequality that we will use in the rest of this paper. To simplify notations we will write $M$ instead 
of $M_N$. We also set œ$ Z^{NM}=(Z^N\otimes I_M, I_N\otimes Y^M) $ and $Z=(z\otimes 1, 1\otimes y)$.

\begin{prop}
	\label{lips}
	
	Let $P\in \A_{d,p+q}$, there are some polynomials $H_P$ and $K_P$ which only depends on $P$ such that for any $N,M$,
	
	\begin{align*}
		\P\Big( &\left|\ \norm{P^*P(X^N\otimes I_M, Z^{NM},{Z^{NM}}^*)} - \E\left[\norm{P^*P(X^N\otimes I_M, Z^{NM},{Z^{NM}}^*)}\right]\ \right| \\
		&\geq \delta + K_P\left(\norm{Z^{NM}}\right)\ e^{-N} \Big) \leq d\ e^{-2N} + e^{-\frac{\delta^2 N}{H_P\left(\norm{Z^{NM}}\right)}} \.
	\end{align*}
	
\end{prop}

\begin{proof}
	We want to use Lemma \ref{saintflour} and Proposition \ref{concentr}. 
	The metric space we will work with is $\R^n$ endowed with 
	the Euclidian norm, and we can take  the function $g$ to be $g:x\mapsto x^2/2$ by 
	Lemma \ref{concentr}. Thus we get that for any $B\subset R^n$, for any $G:R^n\mapsto \R$, if $\gamma=(\gamma_1\etc \gamma_n)$ 
	is a vector of independent centered Gaussian random variables of variance $1$, then for all $\delta>0$,
	\begin{equation}
		\label{interinter}
		\P\left( G(\gamma) - \E[G(\gamma)] \geq \delta +\delta(G) \right) \leq e^{-\frac{\delta^2}{2 (|G|_{\mathcal{L}}^B)^2}} \.
	\end{equation}
	
	\noindent If $0\in B$ as it will be the case later on, we have $\delta(G) \leq \E[ \1_{\gamma\notin B} ( |G(\gamma)| + \sup_{u\in B} |G(u)| + |f|_{\mathcal{L}}^B \norm{\gamma}_2 ) ] $. We set $ B_N = \left\{ \forall i, \norm{X_i^N}\leq D \right\}$ where $D$ was chosen thanks to \ref{bornenorme} such that for any $N$ and $i$,
	
	\begin{equation}
	\label{boundary}
	\P\left(\norm{X^N_i}\geq D \right) \leq e^{- 2 N} \.
	\end{equation}
	
	\noindent Thus we have $\P(B_N^c)\leq d\ e^{-2N}$. With $\gamma$ the vector of size $dN^2$ which consists of the union of 
	$(\sqrt{N}X^{N,i}_{s,s})_{i,s}$, $(\sqrt{2N}\ \Re{X^{N,i}_{s,r}})_{s<r,i}$ and  $(\sqrt{2N}\ \Im{X^{N,i}_{s,r}})_{s<r,i}$, we set 
	$G_N(\gamma) = \norm{P^*P(X^N\otimes I_M, Z^{NM},{Z^{NM}}^*)}$. One can see that on $B_N$ we can find a polynomial 
	$H'_P$ such that for any $N$ and $Z^{NM}$,
	
	$$ |G_N(\gamma) - G_N(\tilde{\gamma})| \leq H'_P\left( \norm{Z^{NM}} \right) \sum_i \norm{X^N_i - \tilde{X}^N_i } \, $$
	
	\noindent where $\norm{.}$ is the operator norm. Besides
	
	\begin{align*}
	\sum_i \norm{X^N_{i} - \tilde{X}^N_{i}} &\leq \sum_i \tr_N\left( (X^N_{i} - \tilde{X}^N_{i})^* (X^N - \tilde{X}^N_{i}) \right)^{1/2}
	\leq \frac{2^d}{\sqrt{N}} \norm{\gamma-\tilde{\gamma}}_2 \.
	\end{align*}
	
	\noindent Thus, on $B_N$, $G_N$ has Lipschitz constant $2^d H'_P\left( \norm{Z^{NM}} \right) N^{-1/2} $. Consequently with \eqref{interinter}, we get that
	\begin{equation*}
	\P\left( G_N(\gamma) - \E[G_N(\gamma)] \geq \delta +\delta(G_N) \right) \leq e^{-\frac{\delta^2 N}{2^{d+1} H'_P\left( \norm{Z^{NM}} \right)^2}} \.
	\end{equation*}
	
	\noindent Therefore, we set $H_P = 2^{d+1} H'_P$, we also have that $ \norm{\gamma}_2^2 = N \sum_i \tr_N((X_i^N)^2)$. Consequently we have some polynomial $K_P'$ such that,
	$$\delta(G) \leq \E\left[ \1_{\left\{ \exists i, \norm{X_i^N}> D \right\}} \left( |G_N(\gamma)| + K_P'(\norm{Z^{NM}}) + 2^d H'_P\left( \norm{Z^{NM}} \right) N^{1/2} \sqrt{\sum_i \norm{X_i^N}^2} \right) \right] $$
	
	\noindent Hence the conclusion thanks to Proposition \ref{bornenorme} and our choice of $D$ in equation \eqref{boundary}.
	
\end{proof}

We   can now prove Theorem \ref{strongconv}. Firstly, we can assume that $Z^{N}$ and $Y^M$ are deterministic matrices by Fubini's 
Theorem. The convergence in distribution is a well-known theorem, we refer to \cite{alice}, Theorem 5.4.5. We set $g$ a 
$\mathcal{C}^{\infty}$ function which takes value $0$ on $(-\infty,1/2]$ and value $1$ on $[1,\infty)$, and belongs  to $[0,1]$ otherwise. 
Let us define $f_{\varepsilon}:t\mapsto g(\varepsilon^{-1} (t - \norm{PP^*(x\otimes 1, Z,Z^*)}))$. By Theorem \ref{imp1}, 
 there exists a constant $C$ which only depends on $P$, $\sup_M \norm{Y^M}$ and $\sup_N \norm{Z^{N}}$ (which is finite thanks to 
the strong convergence assumption on $Z^{N}$) such that,

\begin{align*}
\Bigg| &\E\left[\tr_{MN}\Big(f_{\varepsilon}\left(PP^*\left(X^N\otimes I_M,Z^{NM},{Z^{NM}}^*\right)\right)\Big)\right] - MN \tau_N\otimes\tau_M\Big(f_{\varepsilon}\left(PP^*\left(x\otimes I_M,Z^{NM},{Z^{NM}}^*\right)\right)\Big) \Bigg| \\
&\leq C \varepsilon^{-6} \frac{M^3}{N} \.
\end{align*}
	
\noindent According to Theorem A.1 from \cite{male}, $(x,Z^{N})_{N\geq 1}$ converges strongly in distribution towards $(x,z)$. 
Besides thanks to Lemma \ref{tensorconv} and Corollary 17.10 from \cite{exact}, we have that $(x\otimes I_{M},1\otimes Y^{M})_{M\geq 1}$ 
converges strongly in distribution towards $(x\otimes 1, 1\otimes y)$. In Theorem \ref{strongconv}, we are interested in the situation where 
$Z^{NM} = Z^N\otimes I_M$ or $Z^{NM} = I_N\otimes Y^M$. So, without loss of generality, we restrict ourselves to this kind of $Z^{NM}$. 
We know that $(x\otimes I_M, Z^{NM})$ converges strongly towards $(x\otimes 1, Z)$, but since the support of 
$f_{\varepsilon}$ is disjoint from the spectrum of $PP^*(x\otimes 1,Z,Z^*)$, thanks to Proposition \ref{hausdorff}, for $N$ large enough, 
$\tau_N\otimes\tau_{M}\Big(f_{\varepsilon}(PP^*(x\otimes I_M, Z^{NM},{Z^{NM}}^*))\Big) = 0$ and therefore,
	
\begin{equation}
\label{Neps}
	\E\left[\tr_{MN}\Big(f_{\varepsilon}\left(PP^*\left(X^N\otimes I_M,Z^{NM},{Z^{NM}}^*\right)\right)\Big)\right] \leq C \varepsilon^{-6} \frac{M^3}{N} \.
\end{equation}

\noindent  Hence, using Proposition \ref{bornenorme}, we deduce for $N$ large enough,
	
\begin{align*}
	&\E\left[ \norm{PP^*(X^N\otimes I_M, Z^{NM},{Z^{NM}}^*)} \right] - \norm{PP^*(x\otimes I_M, Z,Z^*)} \\
	&\leq \varepsilon + \int_{\varepsilon}^{\infty}\P\left( \norm{PP^*(X^N\otimes I_M, Z^{NM},{Z^{NM}}^*)} \geq \norm{PP^*(x\otimes I_M, Z,Z^*)} + \alpha \right)\  d\alpha \\
	&\leq \varepsilon + \int_{\varepsilon}^{K} \P\left( \tr_{NM}\left(f_{\alpha}(P(X^N\otimes I_M, Z^{NM},{Z^{NM}}^*))\right) \geq 1 \right)\  d\alpha + C e^{-N}\\
	&\leq \varepsilon + C'\varepsilon^{-5} \frac{M^3}{N} \.
\end{align*}
	
\noindent Finally we get that,

$$ \limsup_{N\to \infty}\ \E\left[ \norm{PP^*(X^N\otimes I_M, Z^{NM},{Z^{NM}}^*)} \right] \leq \norm{PP^*(x\otimes I_M, Z,Z^*)} \.$$
	
Besides, we know thanks to Theorem 5.4.5 of \cite{alice} that if $h$ is a continuous function taking positive values on 
$\left(\norm{PP^*(x\otimes 1, Z,Z^*)}-\varepsilon, \infty \right)$ and taking value $0$ elsewhere. Then 
$\frac{1}{MN}\tr_{MN}(h(PP^*(X^N\otimes I_M, Z,Z^*)))$ converges almost surely
towards $\tau_{\A}\otimes_{\min}\tau_{\B} (h(PP^*(x\otimes 1, Z,Z^*)))$. 
If this quantity is positive for any $h$, then for any $\varepsilon>0$, for $N$ large enough,

$$ \norm{PP^*(X^N\otimes I_M, Z^{NM},{Z^{NM}}^*)} \geq \norm{PP^*(x\otimes 1, Z,Z^*)} - \varepsilon \.$$

\noindent Since $h$ is non-negative and the intersection of the support of $h$ with the spectrum of $PP^*(x\otimes 1, Z,Z^*)$ is non-empty, 
we have that $h(PP^*(x\otimes 1, Z,Z^*)) \geq 0$ and is not $0$. 
Besides, we know that the trace on the space where $z$ is defined is faithful, 
and so is the trace on the $\mathcal{C}^*$-algebra generated by a single semicircular variable, hence by Theorem \ref{freesum}, so is 
$\tau_{\A}$. Thus, since both $\tau_{\A}$ and $\tau_{\B}$ are faithful, by Lemma \ref{faith}, so is $\tau_{\A}\otimes_{\min}\tau_{\B}$ and 
$\tau_{\A}\otimes_{\min}\tau_{\B} (h(PP^*(x\otimes 1, Z,Z^*)))>0$. As a consequence,  almost surely,

\begin{equation}
	\label{liminfps}
	\liminf_{N\to \infty} \norm{P(X^N\otimes I_M, Z^{NM},{Z^{NM}}^*)} \geq \norm{P(x\otimes 1, Z,Z^*)} \.
\end{equation}

\noindent Thanks to Fatou's Lemma, we deduce
$$ \liminf_{N\to \infty}\ \E\left[ \norm{PP^*(X^N\otimes I_M, Z^{NM},{Z^{NM}}^*)} \right] \geq \norm{PP^*(x\otimes I_M, Z,Z^*)} \. $$
	
\noindent We conclude that
\begin{equation}
	\label{limes}
	\lim_{N\to \infty}\ \E\left[ \norm{PP^*(X^N\otimes I_M, Z^{NM},{Z^{NM}}^*)} \right] = \norm{PP^*(x\otimes I_M, Z,Z^*)} \.
\end{equation}

\noindent Let us define the following objects,
$$\varepsilon_N = \left|\E\left[ \norm{PP^*(X^N\otimes I_M, Z^{NM},{Z^{NM}}^*)} \right] - \norm{PP^*(x\otimes I_M, Z,Z^*)}\right| \,$$
$$K = \sup_{N,M\geq 0} K_P\left(\norm{Z^{NM}}\right) + H_P\left(\norm{Z^{NM}}\right) \. $$ 

\noindent $K$ is finite thanks to the strong convergence of the families $Z^N$ and $Y^M$. Then thanks to Proposition \ref{lips}, we have 
that for any $\delta>0$,
\begin{align*}
\P\Big( &\left|\ \norm{P^*P(X^N\otimes I_M, Z^{NM},{Z^{NM}}^*)} - \norm{PP^*(x\otimes I_M, Z,Z^*)}\ \right| \geq \delta + K e^{-N} + \varepsilon_N \Big) \leq d\ e^{-2N} + e^{-\frac{\delta^2 N}{K}} \.
\end{align*}

\noindent Since this is true for any $\delta>0$, by Borel-Cantelli's Lemma, almost surely,
\begin{equation*}
\lim_{N\to \infty} \norm{PP^*(X^N\otimes I_M, Z^{NM},{Z^{NM}}^*)} = \norm{PP^*(x\otimes 1, Z,Z^*)} \.
\end{equation*}
We finally conclude  thanks to the fact that for any $y$ in a $\CC^*$-algebra, $\norm{yy^*} = \norm{y}^2$.

\subsection{Proof of Theorem \ref{concentration}}
We first prove the following estimate that we use multiple times during the proofs.

\begin{lemma}
	\label{meilleurestime}
		
	Let $g$ be a $\mathcal{C}^{\infty}$ function which takes value $0$ on $(-\infty,1/2]$ and value $1$ on $[1,\infty)$, and in $[0,1]$ 
	otherwise. We set $f_{\varepsilon}:t\mapsto g(\varepsilon^{-1} (t - \alpha))$ with $\alpha = \norm{PP^*(x\otimes I_M, 1\otimes Y^M)}$, 
	then there exists a constant $C$ such that for any $\varepsilon>0$ and $N$,
	
	\begin{equation*}
	\E\left[\frac{1}{MN}\tr_{NM}\Big(f_{\varepsilon}(PP^*(X^N\otimes I_M, I_N\otimes Y^M))\Big)\right] \leq C \frac{\varepsilon^{-4}}{N^2}  \.
	\end{equation*}
	
\end{lemma}

\begin{proof}
	
To estimate the above  expectation we use the same method as in the proof of Theorem \ref{strongconv} with a few refinements to have 
an optimal estimate with respect to $\varepsilon$. 
Let $g$ be a $\mathcal{C}^{\infty}$ function which takes value $0$ on 
$(-\infty,1/2]$, $1$ on $[1,\infty)$, and belongs to $[0,1]$ otherwise. We then set 
$f^{\kappa}_{\varepsilon}:t\mapsto g(\varepsilon^{-1} (t - \alpha)) g(\varepsilon^{-1} (\kappa -t)+1)$ with 
$\alpha = \norm{PP^*(x\otimes I_M, 1\otimes Y_M)}$ and $\kappa>\alpha $. Since $g$ has compact support and is sufficiently smooth 
we can apply Theorem \ref{imp2}. Setting  
$h:t\mapsto g(t - \varepsilon^{-1}\alpha) g( \varepsilon^{-1}\kappa +1  -t)=\hat{f^{\kappa}_{\varepsilon}}(\varepsilon t)$, we have

\begin{align*}
	2\pi \int y^4 |\hat{f^{\kappa}_{\varepsilon}}(y)|\ dy &= \int y^4 \left|\int g(\varepsilon^{-1} (t - \alpha)) g(\varepsilon^{-1} (\kappa -t)+1) e^{-\i y t}\ dt \right|\ dy \\
	&= \int y^4 \left|\int h(t) e^{-\i y \varepsilon t}\ \varepsilon dt \right|\ dy \\
	&= \varepsilon^{-4} \int y^4 \left|\int h(t) e^{-\i y t}\ dt \right|\ dy \\
	&\leq \varepsilon^{-4} \int \frac{1}{1+y^2}\ dy \int ( |h^{(4)}(t)| + |h^{(6)}(t)|)\ dt \.
\end{align*}

\noindent The derivatives $h^{(4)}$ and $h^{(6)}$ are uniformly bounded independently of $t$ or $\varepsilon$. Since the support 
of these functions is included in $[\varepsilon^{-1}\alpha,\varepsilon^{-1}\alpha+1]\cup [\varepsilon^{-1}\kappa,\varepsilon^{-1}\kappa +1]$, 
there is a universal constant $C$ such that for any $\varepsilon$ and $\kappa$, 

$$ \int y^4 |\hat{f^\kappa_{\varepsilon}}(y)|\ dy \leq C \varepsilon^{-4} \.$$

\noindent With similar computations we can find a constant $C$ such that for any $\varepsilon$ and $\kappa$, 

\begin{equation}
	\label{eqfourr}
	\int (|y|+y^4) |\hat{f^{\kappa}_{\varepsilon}}(y)|\ dy \leq C \varepsilon^{-4} \.
\end{equation}

\noindent Since the support of $f^{\kappa}_{\varepsilon}$ is disjoint from the spectrum of $PP^*(x\otimes I_M, 1\otimes Y^M)$, for any 
$\varepsilon$ and $N$ one have $\tau\otimes\tau_M\Big(f^{\kappa}_{\varepsilon}(PP^*(x\otimes I_M, 1\otimes Y^M))\Big) = 0$. 
\noindent Consequently thanks to Theorem \ref{imp2}, we have a constant $C$ such that for any $N$, $\varepsilon$ and $\kappa$,

\begin{equation*}
\E\left[\frac{1}{MN}\tr_{NM}\Big(f^{\kappa}_{\varepsilon}(PP^*(X^N\otimes I_M, I_N\otimes Y^M))\Big)\right] \leq C \frac{\varepsilon^{-4}}{N^2}  \.
\end{equation*}

\noindent We define $f_{\varepsilon}:t\mapsto g(\varepsilon^{-1} (t - \alpha))$, then by the monotone convergence theorem, we deduce

\begin{align*}
\E&\left[\frac{1}{MN}\tr_{NM}\Big(f_{\varepsilon}(PP^*(X^N\otimes I_M, I_N\otimes Y^M))\Big)\right] \\
&= \lim_{\kappa\to \infty} \E\left[\frac{1}{MN}\tr_{NM}\Big(f^{\kappa}_{\varepsilon}(PP^*(X^N\otimes I_M, I_N\otimes Y^M))\Big)\right] \.
\end{align*}

\noindent Hence we have a constant $C$ such that for any $N$ and $\varepsilon>0$,

\begin{equation*}
\E\left[\frac{1}{MN}\tr_{NM}\Big(f_{\varepsilon}(PP^*(X^N\otimes I_M, I_N\otimes Y^M))\Big)\right] \leq C \frac{\varepsilon^{-4}}{N^2}  \.
\end{equation*}

\end{proof}

We finally complete  the proof  of Theorem \ref{concentration}. One can view $X^N = (X^N_1\etc X^N_d)$ as the random vector of 
size $dN^2$ which consists of the union of $(\sqrt{N}X^{N,i}_{s,s})_{i,s}$, $(\sqrt{2N}\ \Re{X^{N,i}_{s,r}})_{s<r,i}$ and  
$(\sqrt{2N}\ \Im{X^{N,i}_{s,r}})_{s<r,i}$ which are indeed centered Gaussian random variable of variance $1$ as stated in 
Definition \ref{GUEdef}. Thus we can apply the Poincar\'e inequality (see Proposition \ref{Poinc}) to the function

$$\varphi : X^N \mapsto \frac{1}{MN}\tr_{MN}\left(f_{\varepsilon}(PP^*(X^N\otimes I_M, I_N\otimes Y^M)) \right) \,$$

\noindent and we get

$$\var \left(\varphi\right)\le \frac{1}{(MN)^{2}}\mathbb E[\|\nabla \varphi\|_{2}^{2}]$$

\noindent Besides, as in the proof  of Lemma \ref{grcalc}, if $Q\in \A_{d,p+q}$,

\begin{align*}
&N \norm{\nabla \tr_{MN}\left(Q(X^N\otimes I_M, I_N\otimes Y^M)\right)}_2^2 \\
=& \sum_s \sum_{i,j} \tr_{MN}\Big( D_s Q \ E_{i,j}\otimes I_M \Big) \tr_{MN}\Big( D_s Q \ E_{i,j}\otimes I_M \Big)^* \. \\
\end{align*}

Besides, if $f_k$ is a polynomial with a single variable, then $ D_s f_k(PP^*) = \partial_s (PP^*) \widetilde{\#} f_k'(PP^*) $. Thus, taking 
$f_k$ such that $f_k'$ converges towards $f_{\varepsilon}'$ for the sup norm on the spectrum of $PP^*(X^N\otimes I_M, I_N\otimes Y^M)$, 
we deduce that

$$\var\left(\varphi\right)\leq \frac{1}{M^2 N^3} \sum_s \E\left[ \sum_{i,j} \tr_{MN}\Big( \partial_s (PP^*) \widetilde{\#} f_{\varepsilon}'(PP^*) \ E_{i,j}\otimes I_M \Big) \tr_{MN}\Big( \partial_s (PP^*) \widetilde{\#} f_{\varepsilon}'(PP^*) \ E_{i,j}\otimes I_M \Big)^*  \right] \.
$$

\noindent Now with $A = \partial_s (PP^*) \widetilde{\#} f_{\varepsilon}'(PP^*)$,

\begin{align*}
\sum_{i,j} \tr_{MN}\Big( A \ E_{i,j}\otimes I_M \Big) \tr_{MN}\Big( A \ E_{i,j}\otimes I_M \Big)^* &= \sum_{i,j,k,l} g_j^*\otimes e_k^* A g_i\otimes e_k\ g_i^*\otimes e_l^* A^* g_j\otimes f_l \\
&= \sum_{j,k,l} g_j^*\ (I_N\otimes e_k^*\ A\ I_N\otimes e_k\ I_N\otimes e_l^*\ A^*\ I_N\otimes e_l)\ g_j \\ 
&= \tr_{N} \left( I_N\otimes\tr_{M}(A)\ I_N\otimes\tr_{M}(A^*)\right) \\
&= \tr_{N} \left( I_N\otimes\tr_{M}(A)\ (I_N\otimes\tr_{M}(A))^*\right) \.
\end{align*}

\noindent So by contractivity of the conditional expectation over $\M_N(\C)\otimes I_M$, that is $I_N\otimes \frac{1}{M}\tr_{M}$, we have

\begin{align*}
\sum_{i,j} \tr_{MN}\Big( A \ E_{i,j}\otimes I_M \Big) \tr_{MN}\Big( A \ E_{i,j}\otimes I_M \Big)^* \leq \tr_{MN}(AA^*)\ M \.
\end{align*}

\noindent As a consequence, we find that

$$\var\left(\varphi\right)\leq \frac{1}{N^3 M} \sum_s \E\left[ \tr_{MN}\left( \partial_s (PP^*) \widetilde{\#} f_{\varepsilon}'(PP^*)\ (\partial_s (PP^*) \widetilde{\#} f_{\varepsilon}'(PP^*))^* \right) \right] \.
$$

\noindent Besides, if $U,V$ and $W$ are monomials,

\begin{align*}
\left|\tr_{MN}(Uf_{\varepsilon}'(PP^*)Vf_{\varepsilon}'(PP^*)W) \right| &\leq \sqrt{\tr_{MN}(U{f_{\varepsilon}'}^2(PP^*)U^*)\ \tr_{MN}(Vf_{\varepsilon}'(PP^*)WW^*f_{\varepsilon}'(PP^*)V^*)} \\
&\leq \tr_{MN}({f_{\varepsilon}'}^2(PP^*)) \norm{U} \norm{V} \norm{W} \.
\end{align*}
Therefore there is a constant $C$ depending only on $P$ and $\sup_i \norm{Y^M_i}$ such that

$$\var\left(\varphi\right) \\
\leq \frac{C}{N^2} \E\left[ \prod_s \left( \norm{X_s^N}^{2\deg P} +1 \right) \frac{1}{M N}\tr_{NM}\left( \left| f_{\varepsilon}'(PP^*(X^N\otimes I_M, I_N\otimes Y^M)) \right|^2 \right)\right] \.
$$

\noindent Thanks to Proposition \ref{bornenorme}, we can find $w$ and $\alpha$ such that for any $s$ and $u\geq 0$,

$$ \P\left(\norm{X^N_s}\geq w+u \right) \leq e^{-\alpha u N} \.$$

\noindent There is a constant $C$ independent of $N$ and $\varepsilon$ such that

\begin{equation}
\label{newvar}
\var\left(\varphi\right)\le 
 \frac{C}{N^2} \left( \E\left[ \frac{1}{M N}\tr_{NM}\left( (f_{\varepsilon}')^2(PP^*(X^N\otimes I_M, I_N\otimes Y^M)) \right)\right] + \varepsilon^{-2} e^{- N} \right) \. \nonumber
\end{equation}
We can now apply Theorem  \ref{imp2} to the right hand side of the above equation, noticing that the \eqref{eqfourr} still  holds 
if we replace $f^{\kappa}_{\varepsilon}$ by $(\varepsilon f_{\varepsilon}')^2$. As a consequence, we find an inequality similar the one of 
Lemma \ref{meilleurestime} and thus a constant $C$ such that for any $N$ or $\varepsilon$,

$$ \var\left(\frac{1}{MN}\tr_{NM}(f_{\varepsilon}(PP^*(X^N\otimes I_M, I_N\otimes Y^M)))\right) \leq C\left(\frac{\varepsilon^{-6}}{N^4} + \varepsilon^{-2} e^{- N}\right) \. $$
Therefore, thanks to Lemma \ref{meilleurestime} there exists a constant $C$ such that for any $N\in \N$ and $\varepsilon$ such that 
$\varepsilon^4 > C\frac{M}{N}$,

\begin{align*}
&\P\left( \norm{PP^*(X^N\otimes I_M, I_N\otimes Y^M)} \geq \norm{PP^*(x\otimes I_M, 1\otimes Y^M)} + \varepsilon \right) \\
&\leq \P\left( \frac{1}{MN}\tr_{NM}\left(f_{\varepsilon}(PP^*(X^N\otimes I_M, I_N\otimes Y^M))\right) \geq \frac{1}{MN} \right) \\
&\leq \P\left( \left| \frac{1}{MN}\tr_{NM}\left(f_{\varepsilon}(PP^*)\right) - \E\left[ \frac{1}{MN}\tr_{NM}\left(f_{\varepsilon}(PP^*)\right) \right] \right| \geq \frac{1}{MN} - \frac{C}{N^{2}\varepsilon^4} \right) \\
&\leq C \left(\frac{\varepsilon^{-6}}{N^4} + \varepsilon^{-2} e^{- N}\right) \left( \frac{1}{MN} - \frac{C}{N^{2}\varepsilon^4}  \right)^{-2} \.
\end{align*}

\noindent Let us now set $s = c N^{-1/4}$ with $c$ a constant such that for any $N$,

$$ \frac{1}{MN} - \frac{C}{N^{2}s^4} \geq \frac{1}{2MN} \.$$

\noindent Therefore,  if $x_+ = \max(x,0)$, we have for some constant $C$,

\begin{align*}
& \E\left[ \left(\norm{PP^*(X^N\otimes I_M, I_N\otimes Y^M)} - \norm{PP^*(x\otimes I_M, 1\otimes Y^M)}\right)_+ \right] \\
&= \int_{\R^+} \P\left( \norm{PP^*(X^N\otimes I_M, I_N\otimes Y^M)} \geq \norm{PP^*(x\otimes I_M, 1\otimes Y^M)} + \varepsilon \right)\ d\varepsilon \\
&\leq s + 4C M^2 N^{2} \int_{s}^{\infty} \frac{\varepsilon^{-6}}{N^4} + \varepsilon^{-2} e^{- N} d\varepsilon \\
&\leq s + 4 C M^2 N^{2} (s^{-5}N^{-4} + s^{-1} e^{- N}) \\
&\leq C N^{-1/4} \.
\end{align*}

\noindent On one side, we have

\begin{align*}
&\P\left( \norm{PP^*(X^N\otimes I_M, I_N\otimes Y^M)} - \E\left[\norm{PP^*(X^N\otimes I_M, I_N\otimes Y^M)}\right] \geq \delta + K_P\left(\norm{Y^M}\right)\ e^{-N} \right)	\\
&\geq \P\Big( \norm{PP^*(X^N\otimes I_M, I_N\otimes Y^M)} - \norm{PP^*(x\otimes I_M, 1\otimes Y^M)} \\
&\quad\quad\quad\quad \geq \delta + K_P(\norm{Y^M}) e^{-N} + \E\left[ \left(\norm{PP^*(X^N\otimes I_M, I_N\otimes Y^M)} - \norm{PP^*(x\otimes I_M, 1\otimes Y^M)}\right)_+ \right] \Big)	\\
&\geq \P\left( \left| \norm{P(X^N\otimes I_M, I_N\otimes Y^M)} - \norm{P(x\otimes I_M, 1\otimes Y^M)} \right| \geq \frac{ \delta + C N^{-1/4}}{\norm{P(x\otimes I_M, 1\otimes Y^M)}} \right) \.	\\
\end{align*}

\noindent On the other side, thanks to Proposition \ref{lips}, we have

\begin{align*}
&\P\left( \norm{PP^*(X^N\otimes I_M, I_N\otimes Y^M)} - \E\left[\norm{PP^*(X^N\otimes I_M, I_N\otimes Y^M)}\right] \geq \delta + K_P\left(\norm{Y^M}\right)\ e^{-N} \right)	\\
&\leq e^{-\frac{\delta^2}{H_P(\norm{Y^M})} N} + d e^{-2N} \.
\end{align*}

\noindent Hence we can find constants $K$ and $C$ such that for any $N\in\N$ and $\delta>0$,

$$ \P\left( \norm{P(X^N\otimes I_M, I_N\otimes Y^M)} - \norm{P(x\otimes I_M, 1\otimes Y^M)} \geq \delta + C N^{-1/4} \right) \leq e^{-K \delta^2 N} + d e^{-2N} \.$$

\noindent And we get \eqref{upper} by replacing $\delta$ by $N^{-1/4} \delta$. \\

The other inequality is trickier because we need to study the spectral measure of $PP^*(x\otimes I_M, 1\otimes Y^M)$, which is far from easy. 
We mainly rely on the Theorem 1.1 from \cite{SS15}. We summarize the part of this theorem which is interesting for us in the proposition 
below.

\begin{prop}\label{tail}
	Let $x = (x_1 \etc x_d)$ be a system of free semicircular variable, $p_{i,j}\in \C\langle X_1 \etc X_d \rangle$ be such that 
	$S = (p_{i,j}(x))_{i,j}$ is self-adjoint with spectral measure $\rho$ with support $K$. Then there exists a finite subset $A\subset \R$ 
	such that if $I$ is a connected component of $\R\backslash A$, then either $\rho_{|I} = 0$, or $I\subset K$. In the second situation 
	there exists an analytic function $g$ defined for some $\delta>0$ on 
	
	$$ W := \left\{ z\in\C \middle|\ |\Im z|< \delta \right\} \setminus \bigcup_{a\in A} \left\{ a-\i t \middle|\ t\in\R^+ \right\} $$
	
	\noindent such that for each $a\in A$, there exist $N\in\N$ and $\epsilon>0$ such that $(z-a)^N g(z)$ admits an expansion on 
	$W\cap \left\{ z\in\C \middle|\ |z-a|< \epsilon \right\} $ as a convergent powerseries in $r_N(z-a)$ where $r_N(z)$ is the analytic 
	$N^{th}$-root of $z$ defined with branch $C\setminus \left\{ -\i t \middle|\ t\in\R^+ \right\}$. Then $\Im g_{|I}$ is the probability 
	density function of $\rho_{|I}$.
\end{prop}

What this means for us is that at the edge of the spectrum of $PP^*(x\otimes I_M, 1\otimes Y^M)$, either we have an atom or the density of 
the spectral measure decays like $\frac{1}{|x-a|^r}$ with $r\in\mathbb Q$ when approaching $a$. Consequently we can find $\beta\geq 0$ 
such that if $\rho$ is the spectral measure of $PP^*(x\otimes I_M, 1\otimes Y^M)$ then for $\varepsilon>0$ small enough,

$$ \rho\left(\left[\norm{PP^*(x\otimes I_M, 1\otimes Y^M)} - \varepsilon, \infty \right]\right) \geq \varepsilon^{\beta} \.$$

\noindent Consequently if once again $g$ is a $\mathcal{C}^{\infty}$ function which takes value $0$ on $(-\infty,0]$, $1$ on $[1/2,\infty)$, 
and belongs to $(0,1]$ otherwise. We then take  $f_{\varepsilon}:t\mapsto g(\varepsilon^{-1} (t - \norm{PP^*(x\otimes I_M, 1\otimes Y^M)} + \varepsilon))$ for some $\varepsilon\ge 0$. Then

\begin{align*}
&\P\left( \norm{PP^*(X^N\otimes I_M, I_N\otimes Y^M)} \leq \norm{PP^*(x\otimes I_M, 1\otimes Y^M)} - \varepsilon \right) \\
&= \P\left( \frac{1}{MN} \tr_{NM}\left( f_{\varepsilon}(PP^*(X^N\otimes I_M, I_N\otimes Y^M))\right) = 0 \right) \\
&\leq \P\left( \left| \frac{1}{MN} \tr_{NM}\left( f_{\varepsilon}(PP^*)\right) - \E\left[\frac{1}{MN} \tr_{NM}\left( f_{\varepsilon}(PP^*)\right)\right] \right| \geq \E\left[\frac{1}{MN} \tr_{NM}\left( f_{\varepsilon}(PP^*)\right)\right] \right) \\
&\leq \frac{\var\left( \frac{1}{MN} \tr_{NM}\left( f_{\varepsilon}(PP^*)\right) \right)}{\E\left[ \frac{1}{MN} \tr_{NM}\left( f_{\varepsilon}(PP^*)\right) \right]^2} \.
\end{align*}

\noindent Thanks to \eqref{newvar}, we have

\begin{align*}
\var\left( \frac{1}{MN} \tr_N\left( f_{\varepsilon}(PP^*)\right) \right) &\leq \frac{C}{N^2} \left( \E\left[ \frac{1}{M N}\tr_{NM}\left( (f_{\varepsilon}')^2(PP^*) \right)\right] + \varepsilon^{-2} e^{- N} \right) \\
&\leq \frac{C}{N^2} \left( \norm{f_{\varepsilon}'}^2 + \varepsilon^{-2} \right)\leq \frac{C'}{N^2} \varepsilon^{-2} \.
\end{align*}

\noindent On the contrary, with the same kind of computations which let us get Lemma \ref{meilleurestime}, we can find constants $C$ and $K$ such that

\begin{align*}
\E\left[ \frac{1}{MN} \tr_{NM}\left( f_{\varepsilon}(PP^*)\right) \right] &\geq \tau\otimes\tau_M(f_{\varepsilon}(PP^*)) - C \frac{\varepsilon^{-4}}{N^2} \\
&\geq \rho\left(\left[\norm{PP^*(x\otimes I_M, 1\otimes Y^M)} - \varepsilon/2, \infty \right]\right) - C \frac{\varepsilon^{-4}}{N^2} \geq K \min(1,\varepsilon)^{\beta} - C \frac{\varepsilon^{-4}}{N^2} \.
\end{align*}

\noindent Therefore we find finite constants $C$ and $K$ such that
$$
\P\left( \norm{PP^*(X^N\otimes I_M, I_N\otimes Y^M)} \leq \norm{PP^*(x\otimes I_M, 1\otimes Y^M)} - \varepsilon \right) 
\leq  \frac{K}{N^2 \varepsilon^2} \left( \min(1,\varepsilon)^{\beta} - C \frac{\varepsilon^{-4}}{N^2} \right)^{-2} \.$$
Now we fix $r = c N^{-1/(3+\beta)}$, with $c$ constant such that for any $N$,

$$ \min(1,r)^{\beta} - \frac{C}{N^2r^{4}} \geq \frac{\min(1,r)^{\beta}}{2} \.$$

\noindent Then, we have

\begin{align*}
& \E\left[ \left( \norm{PP^*(x\otimes I_M, 1\otimes Y^M)} - \norm{PP^*(X^N\otimes I_M, I_N\otimes Y^M)}\right)_+ \right] \\
&= \int_{\R^+} \P\left( \norm{PP^*(X^N\otimes I_M, I_N\otimes Y^M)} \leq \norm{PP^*(x\otimes I_M, 1\otimes Y^M)} - \varepsilon \right) d\varepsilon \\
&\leq r + 4K N^{-2} \int_{r}^{\infty} \varepsilon^{-2}\min(1,\epsilon)^{-2\beta} d\varepsilon \leq r + 4K N^{-2} (r^{-1-2\beta}+1) \\
&\leq C N^{-1/(3+\beta)} \.
\end{align*}
We deduce the following bound
\begin{align*}
&\P\left( \norm{PP^*(X^N\otimes I_M, I_N\otimes Y^M)} - \E\left[\norm{PP^*(X^N\otimes I_M, I_N\otimes Y^M)}\right] \leq -\delta - K_P\left(\norm{Y^M}\right)\ e^{-N} \right)	\\
&\geq \P\Big( \norm{PP^*(X^N\otimes I_M, I_N\otimes Y^M)} - \norm{PP^*(x\otimes I_M, 1\otimes Y^M)} \\
&\quad\quad\quad\quad \leq -\delta - K_P(\norm{Y^M}) e^{-N} - \E\left[ \left(\norm{PP^*(x\otimes I_M, 1\otimes Y^M)} - \norm{PP^*(X^N\otimes I_M, I_N\otimes Y^M)}\right)_+ \right] \Big)	\\
&\geq \P\Big( \norm{PP^*(X^N\otimes I_M, I_N\otimes Y^M)} - \norm{PP^*(x\otimes I_M, 1\otimes Y^M)} \leq -\delta - C N^{-1/(3+\beta)} \Big) \.
\end{align*}

\noindent Since on the event $\left\{\forall i, \norm{X_i^N}\leq D\right\}$ with $D$ as in \eqref{boundary}, we have

\begin{align*}
	&\norm{PP^*(X^N\otimes I_M, I_N\otimes Y^M)} - \norm{PP^*(x\otimes I_M, 1\otimes Y^M)} \\
	&\leq \left(\norm{P(X^N\otimes I_M, I_N\otimes Y^M)} - \norm{P(x\otimes I_M, 1\otimes Y^M)}\right) \left( J_P(\norm{Y^M}) + \norm{P(x\otimes I_M, 1\otimes Y^M)} \right) \,
\end{align*}
we deduce that
\begin{align*}
&\P\left( \norm{PP^*(X^N\otimes I_M, I_N\otimes Y^M)} - \E\left[\norm{PP^*(X^N\otimes I_M, I_N\otimes Y^M)}\right] \leq -\delta - K_P\left(\norm{Y^M}\right)\ e^{-N} \right)	\\
&\geq \P\Big( \norm{PP^*(X^N\otimes I_M, I_N\otimes Y^M)} - \norm{PP^*(x\otimes I_M, 1\otimes Y^M)} \leq -\delta - C N^{-1/(3+\beta)} \text{ and } \forall i, \norm{X_i^N}\leq D \Big) \\
&\geq \P\left( \norm{P(X^N\otimes I_M, I_N\otimes Y^M)} - \norm{P(x\otimes I_M, 1\otimes Y^M)} \leq \frac{ -\delta - C N^{-1/(3+\beta)}}{J_P(\norm{Y^M}) + \norm{P(x\otimes I_M, 1\otimes Y^M)}} \right) \\
&\quad\quad - \P(\exists i, \norm{X_i^N}\geq D)	\\
&\geq \P\left( \norm{P(X^N\otimes I_M, I_N\otimes Y^M)} - \norm{P(x\otimes I_M, 1\otimes Y^M)} \leq \frac{ -\delta - C N^{-1/(3+\beta)}}{J_P(\norm{Y^M}) + \norm{P(x\otimes I_M, 1\otimes Y^M)}} \right) \\
&\quad\quad - d e^{-2N} \.	\\
\end{align*}

\noindent On the other side thanks to Proposition \ref{lips}, we have

\begin{align*}
&\P\left( \norm{PP^*(X^N\otimes I_M, I_N\otimes Y^M)} - \E\left[\norm{PP^*(X^N\otimes I_M, I_N\otimes Y^M)}\right] \leq -\delta - K_P\left(\norm{Y^M}\right)\ e^{-N} \right)	\\
&\leq d\ e^{-2N} + e^{-\frac{\delta^2 N}{H_P\left(\norm{Y^{M}}\right)}} \.
\end{align*}

\noindent Hence we can find constants $K$ and $C$ such that for any $\delta>0$,

$$ \P\left( \norm{P(X^N\otimes I_M, I_N\otimes Y^M)} - \norm{P(x\otimes I_M, 1\otimes Y^M)} \leq -\delta - C N^{-1/(3+\beta)} \right) \leq e^{-K \delta^2 N} + 2d\ e^{-2N} \. $$

\noindent And we get \eqref{lower} by replacing $\delta$ by $N^{-1/(3+\beta)} \delta$. \\

\section*{Acknowledgements}
B. C. was partially funded by JSPS KAKENHI 17K18734, 17H04823, 15KK0162. F. P. benefited also from the aforementioned Kakenhi  
grants and a MEXT JASSO fellowship. A. Guionnet and F. Parraud were partially supported by Labex Milyon (ANR-10-LABX-0070) of Universit\'e� de Lyon.
The authors would like to thank Narutaka Ozawa for supplying reference \cite{ozabr} for Lemma 4.2.

\bibliographystyle{abbrv}


\end{document}